\documentclass[12pt]{amsart}
\usepackage[latin9]{inputenc}
\usepackage{color}
\usepackage{amsfonts}
\usepackage{amssymb}
\usepackage{amsmath}
\usepackage{amsthm}
\usepackage{geometry}
\usepackage{verbatim}
\usepackage[toc,page]{appendix}
\DeclareMathOperator\Gal{\textrm{Gal}}
\DeclareMathOperator\Aut{\textrm{Aut}}
\DeclareMathOperator\ld{\textrm{ld}}

\begin{document}
\providecommand{\keywords}[1]{\textbf{\textit{Keywords: }} #1}
\newtheorem{theorem}{Theorem}[section]
\newtheorem{lemma}[theorem]{Lemma}
\newtheorem{prop}[theorem]{Proposition}
\newtheorem{kor}[theorem]{Corollary}
\theoremstyle{definition}
\newtheorem{defi}{Definition}[section]
\theoremstyle{remark}
\newtheorem{remark}{Remark}
\newtheorem{problem}{Problem}
\newtheorem{question}{Question}
\newtheorem{conjecture}{Conjecture}
\newtheorem{condenum}{Condition}

\newcommand{\cc}{{\mathbb{C}}}   
\newcommand{\ff}{{\mathbb{F}}}  
\newcommand{\nn}{{\mathbb{N}}}   
\newcommand{\qq}{{\mathbb{Q}}}  
\newcommand{\rr}{{\mathbb{R}}}   
\newcommand{\zz}{{\mathbb{Z}}}  
\newcommand{\fp}{{\mathfrak{p}}}
\newcommand{\fP}{{\mathfrak{P}}}
\newcommand{\ra}{{\rightarrow}}
\newcommand{\divides}{\,|\,}

\newcommand\DN[1]{{{#1}}}
\newcommand\JK[1]{{{#1}}}

\title{The Hilbert--Grunwald specialization property  over number fields}
\author{Joachim K\"onig}
\author{Danny Neftin}
%\dedicatory{Dedicated to Moshe Jarden on his 80th birthday}

\begin{abstract}
Given a finite group $G$ and a number field $K$, we investigate the following question: Does there exist a Galois extension $E/K(t)$ with group $G$ %(or equivalently, a Galois cover $f:X\to \mathbb{P}^1_K$ with group $G$) 
whose set of specializations yields solutions to all Grunwald problems for the group $G$, outside a finite set of primes? Following previous work,  such a Galois extension would be said to have the ``Hilbert--Grunwald property". In this paper we reach a complete  classification of groups $G$ which admit \JK{an extension} with the Hilbert--Grunwald property over fields such as $K=\mathbb{Q}$.
\JK{We thereby also complete the determination of the ``local dimension" of finite groups over $\mathbb{Q}$}. %, a project whose main part was achieved in \cite{KN}.}
\end{abstract}
\maketitle

\section{Introduction and main results}

Given a number field $K$, a finite group $G$, a finite set $\mathcal{S}$ of primes of $K$ and for each $\fp\in \mathcal{S}$ a Galois extension $L^{(\fp)}/K_\fp$ whose Galois group embeds into $G$, the {\it Grunwald problem} $(G,(L^{(\fp)}/K_\fp), \fp\in \mathcal{S})$ is the question whether there exists a Galois extension of $K$ with group $G$ whose completion at $\fp$ equals $L^{(\fp)}/K_\fp$ for each $\fp\in \mathcal{S}$. Such an extension is then called a {\it solution} to the underlying Grunwald problem.
Following \cite{KLN} and \cite{KN}, we say that  
% a $K$-regular $G$-extension $E/K(t)$
 a Galois extension $E/F$ of (transcendence degree $d\ge 1$) function fields over a number field $K$ 
 %(or equivalently, a Galois cover $f:X\to Y$ of $d$-dimensional quasi-projective varieties over $K$) 
 with group $G$  has the {\it Hilbert--Grunwald property} if the following holds:

(HG) There exists a finite set $\mathcal{S}_0$ of primes of $K$ such that every Grunwald problem $(G,(L^{(\fp)}/K_\fp), \fp\in \mathcal{S})$, with $\mathcal{S}$ disjoint from $\mathcal{S}_0$, has a solution inside the set of specializations $E_{t_0}/K$ of $E/F$ at degree-$1$ places $t_0$ of $K$.

The terminology was introduced (with a more restrictive meaning) in \cite{DG}, which proved the following:\\
Given any $K$-regular
%\footnote{I.e., with $K$ algebraically closed in $E$.} 
 $G$-extension $E/K(t)$, there exists a finite set $\mathcal{S}_0$ of primes of $K$ such that for every finite set $\mathcal{S}$ of primes disjoint from $\mathcal{S}_0$ and every {\it unramified} Grunwald problem on $\mathcal{S}$ (i.e., where all the prescribed local extensions are unramified), there exists a solution inside the set of specializations of $E/K(t)$.

Extending this property 
%(which might be called ``unramified Hilbert--Grunwald property" in our terminology) 
to the case of ramified Grunwald problems meets some obvious obstacles. These come (among others) from the so-called Specialization Inertia Theorem (cf.\ Proposition \ref{beckmann}), which, as a special case, implies that primes $\fp$  of $K$ for which no branch point of $E/K(t)$ is $K_\fp$-rational cannot ramify in any specialization of $E/K(t)$ (with possibly finitely many exceptional $p$). %This implies, loosely speaking, that a $G$-extension $E/K(t)$ whose branch points are of ``too high degree" over $K$, cannot yield solutions to all Grunwald problems via specialization. 
%This vague statement was made precise in 
Using this fact, \cite{Leg_Pisa}  
%whose main result essentially amounts to saying that for 
shows \footnote{This is a consequence of Theorem 3.1 and Proposition 3.5 of \cite{Leg_Pisa}, together with the Specialization Inertia Theorem.} that for every group $G$ which occurs as 
 the Galois group of a $K$-regular extension $E/K(t)$, one can construct another $K$-regular $G$-extension $\tilde{E}/K(t)$ which does not have the Hilbert--Grunwald property. 
 %, even though not explicitly stated as such in \cite{Leg_Pisa}).

In view of this, it is natural to investigate the Hilbert--Grunwald property not with respect to a given extension, but with respect to merely a given {\it group} (and number field), i.e., to ask the following:

\begin{question}
\label{ques:hg2}
Given a number field $K$, for which groups $G$ does there exist a Galois extension $E/F$ of transcendence degree $1$ over $K$ with group $G$ which has the Hilbert--Grunwald property?
\end{question}

\iffalse the following:

\begin{question}
\label{ques:hg}
\fi

%whether, given a number field $K$ and a finite group $G$, there {\it exists} a Galois extension $E/F$ of transcendence-degree-$1$ number fields over $K$ with group $G$, such that $E/F$ has the Hilbert-Grunwald property. 
%(In this case, following \cite{KN}, we say that $G$ is of {\it Hilbert-Grunwald dimension} $1$ over $K$; for short: $\textrm{hgd}_K(G)=1$).

\JK{In this paper,} we will answer Question \ref{ques:hg2} in full for large classes of number fields $K$ (in particular, including the case $K=\mathbb{Q}$). 
Concretely we show:
\begin{theorem}
\label{thm:hgmain}
Let $K\subset \mathbb{R}$ be a real number field such that the cyclotomic extensions $K(\zeta_p)$, with $p$ running through the prime numbers, are pairwise distinct. Then the following are equivalent:
\begin{itemize}
    \item[i)] There exists a $K$-regular Galois extension $E/K(t)$ with group $G$ possessing the Hilbert--Grunwald property.
    \item[ii)] $G$ is either a cyclic group of order $2$ or an odd prime power; or $G$ is a Frobenius group whose kernel and complement both are cyclic groups of order either $2$ or an odd prime power.
\end{itemize}
\end{theorem}

\iffalse
%Stronger version, to be moved into Section 4!
\begin{theorem}
\label{thm:hgmain_strong}
Let $K\subset \mathbb{R}$ be a real number field such that the cyclotomic extensions $K(\zeta_p)$, with $p$ running through the prime numbers, are pairwise distinct. Then the following are equivalent:
\begin{itemize}
\item[i)] $\textrm{ld}_K(G)=1$.
\item[ii)] $\textrm{hgd}_K(G) = 1$.
\item[iii)] There exists a $K$-regular Galois extension $E/K(t)$ with group $G$ possessing the Hilbert-Grunwald property.
\item[iv)] $G$ is either a cyclic group of order $2$ or an odd prime power; or $G$ is a Frobenius group whose kernel and complement both are cyclic groups of order $2$ or an odd prime power.
\end{itemize}
\end{theorem}
\fi

\begin{remark}
%\begin{itemize}
%\item[a)] About the assumptions on the field $K$: \marginpar{Insert.}
%    \item[b)]
Recall that a Frobenius group $G$ is a transitive permutation group whose non-\JK{identity} elements all have either no or one fixed point, with the latter case occurring at least once. It is well known that the set of fixed point free elements of a Frobenius group together with the identity forms a normal subgroup $K$, called the Frobenius kernel, and that $G=K\rtimes H$, where $H$ is a point stabilizer, also known as Frobenius complement. In particular, it is not difficult to see that a semidirect product $G=K\rtimes H$ is Frobenius with complement $H$ if and only if $C_G(k) \cap H = \{1\}$ for every $k\in K\setminus\{1\}$. With that in mind, a Frobenius group with cyclic kernel and complement of odd prime power order (or order $2$) is the same as a semidirect product $C_P\rtimes C_Q$ of two cyclic groups of prime power order, where $C_Q$ acts as a group of automorphisms not only on $C_P$ but also on its prime-order subgroup (call that prime $p$). Since $\Aut(C_P)$ and $\Aut(C_p)$ are cyclic, the latter of order $p-1$, which is also the prime-to-$p$ part of the further, it follows that our condition is equivalent to $G\cong C_P\rtimes C_Q$, where $P$ and $Q$ are of {\it coprime} prime power order (or order $2$) and $C_Q$ acts as a group of automorphisms on $C_P$. This is how Condition b) above is worded and used in Section \ref{sec:proofmain}.
\iffalse, and indeed the notion of a Frobenius group is not necessary to state our result, although it does have a more elegant ring to it than the fully explicit description.
\fi
%\end{itemize}
\end{remark}

\JK{Over more general number fields, determining the precise list of groups $G$ with the Hilbert--Grunwald property can vary in somewhat delicate ways. Just to give one example, over a field $K$ containing the $n$-th roots of unity, the extension $K(t^{1/n})/K$ is Galois with group $C_n$ and is a {\textit{generic}} extension for this group, meaning that {\it all} $C_n$-extensions, and in particular all possible local behaviors, occur among its specializations. In particular, it has the Hilbert--Grunwald property over $K$, whereas Theorem \ref{thm:hgmain} shows that, for most $n$, there is no such extension over fields such as $\mathbb{Q}$. 
 We will see in the proof of Theorem \ref{thm:hgmain} that the implication ii)$\Rightarrow$i) always holds, see Theorem \ref{thm:cyclic_odd}, whereas conversely a group with a non-cyclic abelian subgroup can never have (a transcendence degree-$1$ extension with) the Hilbert--Grunwald property, see Corollary \ref{cyclic_dec_group}.
 Below, we give one more class of fields over which the groups with the Hilbert--Grunwald property are  (almost-)completely determined:
\begin{theorem}
\label{thm:qi} 
Let $K$ a number field containing 
%Let $K$ be a totally complex field containing 
$\sqrt{-1}$ such that the cyclotomic extensions $K(\zeta_p)$, with $p$ running through the prime numbers, are pairwise distinct.
%and having distinct $p$-th cyclotomic extensions, and let $G$ be a finite group. 
Suppose $G$ is a finite group for which there exists a $K$-regular $G$-extension $E/K(t)$ with the Hilbert--Grunwald property. Then $G$ is either a cyclic group of prime power order, a Frobenius group whose kernel and complement are both cyclic of prime power order, or a generalized quaternion group.
Moreover, the converse holds as long as $G$ possesses a generic extension over $K$.\footnote{The existence of generic extensions over number fields containing $\sqrt{-1}$ is known for all of the above groups, except possibly the generalized quaternion groups of order $2^k$ ($k\ge 5$); cf.\ \cite[Theorem 2.1]{S82} for the cyclic groups, and \cite[Theorem 1.4]{Kang} for the generalized quaternion groups of order $8$ and $16$.}
\end{theorem}}

Question \ref{ques:hg2} and Theorems \ref{thm:hgmain} and \ref{thm:qi} should be viewed in the broader context of a series of papers investigating (in various senses) the ``arithmetic-geometric complexity" of the Galois theory of a given finite group over a given field $K$. Specifically, the following question is raised in \cite{Koe_GW} (and formalized in \cite{KN}): What is the minimal %dimension of algebraic varieties $X,Y$ over a given number field $K$ such that there exists a Galois cover $X\to Y$ with group $G$ 
integer $d$ such that there exist finitely many $G$-extensions $E_i/F_i$ ($i=1,\dots, r$) of \'etale algebras of transcendence degree $\le d$ over $K$ (equivalently,  Galois extensions of transcendence degree-$d$ fields $E_i/F_i$ with Galois group embedding into $G$) whose specializations provide solutions to all Grunwald problems for $G$, possibly outside some finite set of primes? This dimension is called the {\it Hilbert--Grunwald dimension} $\textrm{hgd}_K(G)$ of $G$ over $K$. In these terms, Question \ref{ques:hg2} asks to classify all groups of Hilbert--Grunwald dimension $1$. 
In particular, any result of the form ``There is no \JK{transcendence degree $1$ extension $E/F$} possessing the Hilbert--Grunwald property" is first and foremost a statement about the complexity of the Galois theory of the group $G$: in this case, the local Galois theory of $G$ (i.e., over completions of the field $K$) is too complex to be captured by a one-dimensional object (i.e., covers of curves). 

The quantity $\textrm{hgd}_K(G)$ is intimately connected with the {\it local dimension} $\textrm{ld}_K(G)$ of $G$ over $K$ (see again \cite{KN}), defined as the minimal 
%dimension of algebraic varieties $X,Y$ over $K$ such that there exists a Galois cover $X\to Y$ with group $G$
integer $d$ such that there exist finitely many $G$-extensions $E_i/F_i$ ($i=1,\dots,r$) of transcendence degree $\le d$ over $K$ such that, for all but finitely many completions $K_{\DN{\fp}}$ of $K$, all Galois extensions $L/K_\fp$ of group embedding into $G$ occur as a specialization of some $E_i/F_i$ (after base change from $K$
 to $K_\fp$).
 %which (after base change to $K_p$) specializes to all Galois extensions $F/K_p$ of group embedding into $G$, for all but finitely many completions $K_p$ of $K$.
  In this context, \cite[Theorem 1]{KN} showed that $\textrm{ld}_K(G)\le 2$ for all groups $G$ and number fields $K$, while furthermore $\textrm{ld}_K(G)\le \textrm{hgd}_K(G)$ holds trivially with equality possible in general % from the respective definitions, with equality remaining plausible 
  \footnote{If equality holds in general, then the inverse Galois problem has an affirmative answer.}. 
  \DN{In fact, in Theorem \ref{thm:hgmain_strong}}, we achieve a full classification of groups $G$ with $\textrm{ld}_K(G)=\textrm{hgd}_K(G)=1$ for fields such as $K=\mathbb{Q}$.

%\subsubsection*{Methods of proof}
One of our main tools is a result on the local behavior of Galois extensions arising via specialization of a given function field extension, shown in \cite{KLN} and recalled in Proposition \ref{mainlemma_kln}. It was noted already in \cite[Theorem 6.2]{KLN} that this result implies a negative answer to Question \ref{ques:hg2} (over any number field $K$) for all groups $G$ possessing a non-cyclic abelian subgroup. In Section \ref{sec:hgp}, we develop this approach further and exhibit some newly discovered obstructions to the Hilbert--Grunwald property (Lemmas \ref{cyclic_dec_group2} and \ref{cyclic_dec_group3}). These, together with the well-known classification of Sylow-cyclic finite groups,  suffice to prove the implication i)$\Rightarrow$ii) in Theorem \ref{thm:hgmain}.
%, namely for all $G$ possessing a non-cyclic abelian subgroup; the main tool was an extension of the Specialization Inertia Theorem to {\it decomposition groups} in specializations, recalled below in Proposition \ref{mainlemma_kln}.

%As manifested in \cite[Remark 6.4]{KLN}, the authors at the time were somewhat under the impression that a full answer to Question \ref{ques:hg2} was out of reach with the currently available tools. The main objective of the present paper is to refute this: indeed, building on and further developing the main results of \cite{KLN}, 

%While the implication ``ii)$\Rightarrow$i)" in Theorem \ref{thm:hgmain} is trivial from the definitions, the proof of ``i)$\Rightarrow$iii)" rests on newly discovered obstructions to a positive answer for Question \ref{ques:hg2}, apart from the existence of a non-cyclic abelian subgroup. 
%See Lemmas \ref{cyclic_dec_group2} and \ref{cyclic_dec_group3}. 
To show that these are the only obstructions, that is, to obtain the implication ii)$\Rightarrow$i), we construct function field extensions with prescribed local behavior using the existence of generic extensnions \cite{S82} and  again  the specialization criteria obtained in \cite{KLN}; cf.\ Section \ref{sec:proofmain}.

On the other extreme, when $K$ contains the $p$-th roots of unity for all $p\divides |G|$, we show the obstructions obtained in Section \ref{sec:hgp} are  also the only obstructions for the Hilbert--Grunwald property for abelian groups, see Theorem \ref{thm:main-cyclic}. Our construction makes use of embedding problems for quadratic extensions into cyclic extensions \cite[\S 16.5]{FJ}, following Geyer--Jansen \cite{GJ} and Arason--Fein--Schacher--Sonn \cite{AFSS}, \DN{and even uses an idea from the construction of polynomials having a root mod $\fp$ for all $\fp$.} %$ in every completion of $K$ but no root in $K$, cf.\ \cite{Sonn}.}

%\marginpar{More methods to be mentioned/ possibly mention Thm.4.6 as well?}}

This paper is dedicated to Moshe Jarden who greatly influenced our research for many years to come. \DN{The authors thank Howard Neur for helpful discussions.}
The second author is grateful for support of the Israel Science Foundation, grant no.\ 353/21, and the support and hospitality of the Institute for Advanced Studies.

%
% Might want to try full classification for abelian groups.
% E.g., cyclic (I,D) = (C_m, C_{mn}) with (m,n)=1 reachable precisely WHEN?
% Proven: If q is any prime divisor of n, then necessarily q-th roots of unity need to be in K. Is this also sufficient?!? 
%

\section{Preliminaries}
\subsection{Completions of number fields}\label{sec:local}
Let $K$ be a number field. For a prime $\fp$ of (the ring of integers $\mathcal O_K$ of) $K$, let $N(\fp):=|O_K/\fp|$ be its norm, and $K_\fp$ the completion of $K$ at $\fp$.
Let $\mu_e$ denote the $e$-th roots of unity in an algebraic closure of $K$, and $\zeta_e\in \mu_e$ a primitive $e$-th root of unity. 
%Let $C_e$ denote the cyclic group of order $e$. We shall henceforth refer to a Galois extension with group $G$ briefly as a $G$-extension. 

\DN{Recall that the Galois group $G^{tr}_\fp$ of the maximal tamely ramified extension $M_\fp$ of $K_\fp$ is (profinitely) generated by two generators $\sigma$ and $\tau$ subject to the single relation $\sigma^{-1}\tau\sigma = \tau^q$, where $q=N(\fp)$. Its inertia subgroup is the cyclic group generated by $\tau$, so that the conjugation action of $\sigma$ on $\langle \tau\rangle/\langle \tau^e\rangle$ is equivalent to its action on $\mu_e$ for all $e$ coprime to $\fp$. In particular, the  abelianization of $G^{tr}_\fp$ is isomorphic to $\mu_{q-1}\times \hat{\mathbb  Z}$, with the inertia group mapping to $\mu_{q-1}$, and furthermore $\mu_{q-1}$ is the group of roots of unity in $K_\fp$ of order coprime to $\fp$.}  
%In particular, if $G=C_e\rtimes C_f$ for coprime $e,f$, there exists a unique $G$-extension of $K_\fp$ with ramification index $e$. 
%%JK: Only for faithful action, and even then it should be "at most one". But I don't think we use the above anymore.

We  note that $K_\fp$ admits a degree $d$ cyclic extension with ramification index $e$ if and only if $\zeta_e\subseteq  K_\fp$. 
Indeed,   
%the inertia subgroup corresponds under this isomorphism to $\mu_{q-1}$. Thus, 
in view of the above isomorphism, the cyclic group $C_d$ of order $d$ is the Galois group of an extension with ramification index $e$, for $e$ coprime to $q$, if and only if there exists an epimorphism $\mu_{q-1}\times \hat{\mathbb Z}\ra C_d$ which maps $\mu_{q-1}$ to $C_e$. Such an epimorphism exists if and only if $e$ divides $q-1$, that is, if and only if $\zeta_e\subseteq K_\fp^\times$. % as claimed. 

%\DN{TODO: Connect the following:}
%Note that a prime $\fp$ of $K$ splits completely in $K(\zeta_{e})$ and has residue degree $d/e$ in $K(\zeta_{d})$ if and only if $\mu_d\cap K_\fp = \mu_e$, and the primes extending $\fp$ in $K(\zeta_{e})$ are inert in $K(\zeta_{d})/K(\zeta_{e})$). 
%Note that a prime $\fp$ splits completely in $K(\zeta_e)$ if and only if $N(\fp)\equiv 1$ mod $e$. 
%the primes $\fp$ which split completely in $K(\zeta_{e})$ are the only ones for which $K_\fp$ can possibly have a cyclic $(e,f)$-extension.

%Note that upon fixing a faithful action of $C_f$ on $C_e$, there is only one $(C_e\rtimes C_f)$-extension of $K_\fp$ with ramification index $e$. 

\subsection{Basics about function field extensions}
%{\bf Joachim: Tried to adapt the below to a sufficiently general setting, in particular also containing the ``specialization" notion for complete fields as occurring in A.1.}
%
\JK{Let $R$ be a Dedekind domain with field of fractions $F$, let $E/F$ be a finite Galois extension and $S$ the integral closure of $R$ in $F$.} 
%Let $K$ be a field, and $E/F$ a finite Galois extension of fraction fields of Dedekind domains, with $F\supseteq K$, and such that $E$ and $\overline{K}$ are contained in some common overfield. We assume that the extension $E/F$ is {\it $K$-regular} (also denoted as {\it regular}, if the base field $K$ is understood), by which we mean that $E\cap \overline{K} = K$ (and a fortiori $F\cap \overline{K} = K$). 
For any prime ideal $\fp$ of $R$ and any prime ideal $\fP$ of $S$ extending $\fp$,   the {\it specialization} of $E/F$ at $\nu$ is the extension $S_{\fP}/R_{\fp}$ of residue fields at the respective prime ideals. Note that  this specialization is independent of the choice of prime ideal extending $\fp$ since $E/F$ is Galois. In particular, in the key case of a rational function field $F=K(t)$ over a field $K$, we denote by $E_{t_0}/K$, for $t_0\in \mathbb{P}^1(K)$, the specialization of $E/K(t)$ at the $K$-rational place $t\mapsto t_0$, \JK{i.e., at the ideal $(t-t_0)$ of $K[t-t_0]$. \footnote{If $t_0=\infty$, one should read $t-t_0$ as $1/t$.}}
\JK{We also denote the set of all specializations $E_{t_0}/K$, for $t_0\in \mathbb{P}^1(K)$, by $Sp(E/K(t))$.}

\JK{If $K\subset F$ is a field which is algebraically closed in $E$, we say that the extension $E/F$ is $K$-regular.}

Now let $K$ be of characteristic zero, and let $E/K(t)$ be a \JK{(not necessarily $K$-regular)} Galois extension with group $G$. Such an extension has finitely many branch points $p_1,...,p_r\in \mathbb{P}^1(\overline{K})$, and associated to each branch point $p_i$ is a unique conjugacy class $C_i$ of $G$, corresponding to the automorphism $(t-p_i)^{1/e_i}\mapsto \zeta (t-p_i)^{1/e_i}$ of the Laurent series field
$\overline{K}(((t-p_i)^{1/e_i}))$, where $e_i$ is minimal such that $E$ embeds into $\overline{K}(((t-p_i)^{1/e_i}))$, and $\zeta$ is a primitive $e_i$-th root of unity. This $e_i$ is the {\it ramification index} at $p_i$, and equals the order of elements in the class $C_i$, \JK{which in term are the generators of inertia subgroups at places extending $t\mapsto p_i$ in $E$}. 
%The class tuple $(C_1,...,C_r)$ is called the {\it inertia canonical invariant} of $L/K(t)$, and the tuple $((p_1,...,p_r), (C_1,...,C_r))$ the {\it ramification structure}.
%\marginpar{Needed??}

%\begin{kor}
%Let $E|k(t)$ be a regular Galois extension with group $G$, let $a\in \mathbb{P}^1(k)$ be a $k$-rational %branch point and let $I\trianglelefteq D \le G$ be the inertia and decomposition group at $a$. Then the %induced map $D/I\to Aut(I)$ is surjective.
%\end{kor}

\subsection{Local behavior of specializations}

Let $K$ be a number field and $E/K(t)$  a \JK{(not necessarily $K$-regular)} Galois extension with group $G$, and $t_0\in \mathbb{P}^1_K$.
We will make extensive use of previous results relating inertia groups, residue fields, etc., at primes $\fp$ in the specialized extension $E_{t_0}/K$ to those in the extension $E/K(t)$. 

In the following proposition we state the well-known criterion of Beckmann \cite[Prop.\ 4.2]{Beckmann} relating ramification in Galois extensions of $K(t)$ to ramification in their specializations. We first introduce some notation:
For $a_0\in \mathbb{P}^1(\overline{K})$, let $f\in K[X]$, resp., $\tilde{f}\in K[X]$ be the minimal polynomial of $a_0$, resp., of $1/a_0$.\footnote{We define the minimal polynomial of infinity as $1$.} 
%Let $\mathfrak{p}$ be a finite prime of $K$. 
Define the {\it intersection multiplicity} $I_\mathfrak{p}(a,a_0)$ as the $\mathfrak{p}$-adic valuation $\nu_\mathfrak{p}(f(a))$ in case $a_0$ is of non-negative $\mathfrak{p}$-adic valuation, resp., as $\nu_\mathfrak{p}(\tilde{f}(1/a))$ otherwise. Obviously, we have $I_\mathfrak{p}(a,a_0)\ne 0$ only for finitely many prime ideals $\mathfrak{p}$ of $K$. 
%
%Beckmann!
\begin{prop}%[Beckmann]
\label{beckmann}
Let $K$ be a number field and $E/K(t)$ be a %$K$-regular
Galois extension with Galois group $G$.
Suppose $a\in K$ is not a branch point of $E/K(t)$. 
%Assume that all branch points of $N|k(t)$ are finite.
For all but finitely many primes $\mathfrak{p}$ of $K$, with the exceptional set depending on $E/K(t)$, 
%the following holds:\\
 %then 
 the following condition is necessary for $\mathfrak{p}$ to be ramified in the specialization $E_{t_0}/K$:
 $$e_i:=I_{\mathfrak{p}}(t_0,t_i)>0 \text{ for some (unique up to conjugation) branch point $t_i$.}$$
Furthermore, the inertia group of a prime extending $\mathfrak{p}$ in the specialization $E_{t_0}/K$ is then conjugate in $G$ to $\langle\tau^{t_i}\rangle$, where $\tau$ is a generator of an inertia subgroup over the branch point $t\mapsto t_i$ of $E/K(t)$.
\end{prop}

\JK{\begin{remark}
\label{rem:beckmann_nonreg}
Note that Proposition \ref{beckmann} is often stated with the additional assumption of $E/K(t)$ being $K$-regular, which is however not necessary after possibly increasing the finite set of exceptional primes. Indeed, if $\kappa$ is the algebraic closure of $K$ inside $E$, one may apply the result for the $\kappa$-regular extension $E/\kappa(t)$ and then deduce it for $E/K(t)$, via noting the following elementary facts: from the definition, inertia groups at a given branch point $t\mapsto t_i$ are the same in $E/K(t)$ and in $E/\kappa(t)$, up to conjugacy in $G$; furthermore, as long as $\mathfrak{p}$ is unramified in $\kappa/K$, intersection multiplicities at $\mathfrak{p}$ and at some (hence any, since $\kappa/K$ is Galois) prime extending $\mathfrak{p}$ in $\kappa$ are identical; finally, inertia groups at $\mathfrak{p}$ in $E_{t_0}/K$ and $E_{t_0}/\kappa$ are also identical (up to conjugacy) as long as $\mathfrak{p}$ is unramified in $\kappa/K$.
\end{remark}}

Next, we deal with residue fields at ramified primes in specializations. For a \JK{Galois} extension $E/K(t)$, a value $t_0\in \mathbb{P}^1_K$ and a prime $\fp$ of $K$, we use the notation $I_{t_0,\fp}$ and $D_{t_0,\fp}$ for the inertia and decomposition group at (a prime extending) $\fp$ in the residue field extension $E_{t_0}/K$.

The following result is used as the main tool in \cite{KLN} (and occurs there in a somewhat more general setting as Theorem 4.1). It relates the residue field, decomposition group etc. at ramified primes in specializations to the respective data at a branch point in the \JK{underlying function field} extension. \JK{Just like Proposition \ref{beckmann}, it is stated only for $K$-regular Galois extensions in \cite{KLN}; however, its proof makes no use of this regularity assumption other than invoking Proposition \ref{beckmann}; thus, the assumption may be dropped by Remark \ref{rem:beckmann_nonreg}.}
\begin{prop}
\label{mainlemma_kln}
%(Conclusions of MainLemma of [KLN])
Let $K$ be a number field and $E/K(t)$ a %$K$-regular
 Galois extension with group $G$.
Let %$t_i\in \overline{k}\cup\{\infty\}$
$t_i\in \overline{K}\cup\{\infty\}$ be a branch point of $E/K(t)$, and set $E':=E(t_i)$ and $K':=K(t_i)$.
Let $\fp$ be a prime of $K$, away from an explicit finite set of ``exceptional" primes depending only on $E/K(t)$, and assume that there exists a prime ${\fp}'$ extending $\fp$ in $K'$ with relative degree $1$.\footnote{Note that, due to the definition of intersection multiplicity, the existence of such a $\fp'$ is a necessary condition for $I_\fp(t_0,t_i)>0$ for any $t_0\in K$.}

Let $t_0\in K$ be a non-branch point such that $I_\fp(t_0,t_i)>0$.
%\footnote{Note that, by the definition of $I_\fp$, such $t_0$ exist as soon as there exists a prime of relative degree $1$ extending $\fp$ in $K(t_i)$.} 
Denote by $I$ and $D$ the inertia and decomposition group at (a fixed place extending) $t\mapsto t_i$ in $E'/K'(t)$.
%\footnote{To be formally correct, $t\mapsto t_i$ is a place of $k(t_i)(t)$, and not of $k(t)$ if $t_i$ is not $k$-rational. It should be understood without further mention that by expressions like ``decomposition group at ... in $E|k(t)$ etc., we mean the corresponding notion in $E(t_i)|k(t_i)(t)$.}
Then the following hold:
\begin{itemize} 
%\item[i)] The inertia group $I_{t_0,\fp}$ equals a subgroup of $I$ (up to conjugacy in $G$).
\item[i)]
 The completion at $\fp'$ of $E'_{t_i}/K'$ is contained in the completion at $\fp$ of $E_{t_0}/K$.\\
In particular, the residue extension at $\fp'$ in $E'_{t_i}/K'$ is contained in the residue extension at $\fp$ in $E_{t_0}/K$.
%\item[iii)(old version)] The Galois group of the completion at $\fp$ of $E_{t_0}$ over (the completion) $k_\fp$ (=the decomposition group at $\fp$ in $E_{t_0}$!) is a subgroup of $Gal(k'((t-t_i,\pi))(y_1(t),...,y_n(t))| k'((t-t_i,\pi)))$, where $\pi$ denotes a uniformizer of $\tilde{\fp}$.
\item[ii)] The \DN{identification} of the decomposition group $D_{t_0,\fp}$  with a subgroup of $D$ (up to conjugacy in $G$) fulfills $\varphi(D_{t_0,\fp})= D_{t_i,\fp'}$, where $\varphi: D\to D/I$ is the canonical epimorphism. 
\end{itemize} 
In particular, if in addition $I_\fp(t_0,t_i)$ is coprime to $e_i:=|I|$, then the following hold:
\begin{itemize}
%\item[iv)] The inertia group $I_{t_0,\fp}$ equals $I$ (up to conjugacy). In particular, the ramification index at $\fp$ in $E_{t_0}/k$ equals $e_i$.
\item[iii)] The decomposition group $D_{t_0,\fp}$ equals $\varphi^{-1}(D_{t_i,\fp'})$ where $\varphi: D\to D/I$ is the canonical epimorphism. Furthermore, the residue extension at $\fp$ in $E_{t_0}/K$ equals the residue extension at ${\fp}'$ in $E'_{t_i}/K'$.
\end{itemize}
\end{prop}
%
%See \cite{KLN} for a proof of Proposition \ref{mainlemma_kln}.

\JK{The above result, which gives necessary conditions on the local behavior of specializations, is complemented by the following (see \cite[Theorem 4.4]{KLN}), showing that all local extensions of a certain form are indeed realizable.

\begin{prop}\label{thm:KLN2}
Let $K$ be a number field, $E/K(t)$ a %$K$-regular
\JK{finite} Galois extension,
%where $K$ is the fraction field of a Dedekind domain $R$
 and $t_1\in \overline{K} \cup \{\infty\}$ a branch point.

Let $D, I$, and $N/K(t_1)$ be the Galois group, inertia group, and residue extension, respectively, of the completion at  $t\mapsto t_1$. Let $\fp$ be a prime away from the finite exceptional set of Proposition  \ref{mainlemma_kln}, 
with a degree $1$ prime $\fp'$ of $K(t_1)$ lying over it,  and define $D'$ as $\varphi^{-1}(D_{t_1,\fp'})$, with the notation of Proposition \ref{mainlemma_kln}.  
Let $T_\fp/K_\fp$ be a $D'$-extension 
%such that $I\leq D'\leq D$
such that $T_\fp^I/K_\fp \cong NK(t_1)_{\fp'}/K(t_1)_{\fp'}$ ($\cong NK_\fp/K_\fp$).

Then there exist infinitely many $t_0\in K$ %of $E/K(t)$ %with $v_\fp(t_0-t_1)>0$  
such that the completion of $E_{t_0}/K$ at $\fp$  is $T_\fp/K_\fp$. 
\end{prop}
}

Finally, we recall the following well-known result about compatibility of weak approximation and Hilbert's irreducibility theorem, cf., e.g., \cite[Proposition 2.1]{PV05}. In particular, it justifies treating the Hilbert--Grunwald property ``one prime at a time".
\begin{lemma}
\label{lem:pv}
Let $K$ be a number field, $E/K(t)$ a %$K$-regular
Galois extension with group $G$ and $\mathcal{S}$ a finite set of primes of $K$. For each $\fp\in \mathcal{S}$, choose a  non-branch point $t_\fp\in K_\fp$ of $E\cdot K_\fp/K_\fp(t)$, and denote the specialization at $t_\fp$ by $L^\fp/K_\fp$. Then there exists $t_0\in K$ such that $E_{t_0}/K$ has Galois group $G$ and has completion $L^\fp/K_\fp$ at each prime $\fp\in \mathcal{S}$.\\
In particular, if for all but finitely many primes $\fp$ of $K$, every Galois extension of $K_\fp$ with group embedding into $G$ is in $Sp(E\cdot K_\fp/K_\fp(t))$, then $E/K(t)$ has the Hilbert--Grunwald property.
\end{lemma}

\subsection{Embedding problems}
Given a Galois extension $L/K$, an {\it embedding problem} is an epimorphism $\pi:G\ra \Gal(L/K)$, and a {\it solution} is a homomorphism $\psi:G_K\ra G$ such that $\pi\circ \psi$ coincides with the restriction map. A solution  is {\it proper} if it is surjective. If $\ker\pi$ is contained in the Frattini subgroup of $G$, then the embedding problem $\pi$ is said to be {\it Frattini}.  In such a case, every solution has to be proper. 

\section{Three obstructions  to local dimension $1$} %$\textrm{\lowercase{ld}}_K(G)=1$}
\label{sec:hgp}
\subsection{Statement of the obstructions}

%\textbf{Due to rewording of Main Theorem, these are now indeed (stronger) obstructions to $ld_K(G)=1$!}

Below, we present three results, Corollary \ref{cyclic_dec_group} and Lemmas \ref{cyclic_dec_group2} and \ref{cyclic_dec_group3}, each of which restricts  the possible $K_\fp$-specializations of a transcendence degree $1$ function field extension $E/F$ over a number field $K$. As will be detailed in the proof of Theorem \ref{thm:hgp}, for any ``sufficiently complicated" group $G$, these restrictions amount to obstructions to the assertion $\textrm{ld}_K(G)=1$; in particular, they ensure that \JK{transcendence degree $1$ extensions over $K$} with group $G$ can \DN{never} have the Hilbert--Grunwald property. The three obstructions are, in short, (1) the existence of a non-cyclic abelian subgroup, (2) the existence of an element whose order is not a prime-power, and (3) the existence of an element of order $4$.\\
All of them are applications of Proposition \ref{mainlemma_kln}; the first one already occurred in the key case $F=K(t)$  in \cite[Theorem 6.3]{KLN}, and extended to arbitrary transcendence-degree-$1$ extensions $E_i/F_i$, $i=1,\dots,r$ in \cite[Theorem B.1]{KN}. 
The other two, however, are new. 

\begin{kor}
\label{cyclic_dec_group}
Let $G$ be a finite group, and $E_i/F_i$, $i=1,\dots, r$, finitely many  Galois extensions with group $G$ of transcendence degree $1$ function fields over $K$. Then there exists a number field $L\supseteq K$ such that the following holds: If $\mathcal{S}$ denotes the set of primes of $K$ which split completely in $L$, then for all $\mathfrak{p}\in \mathcal{S}$ and all degree-$1$ places $t_0$ of $F_i\cdot K_\mathfrak{p}$ ($i=1,\dots, r$), the specialization $(E_i\cdot K_\mathfrak{p})_{t_0}/K_\mathfrak{p}$ has cyclic decomposition group at $\mathfrak{p}$.
%\\
%In fact, up to excluding finitely many primes, one may choose $L$ as the compositum of all residue field extensions at branch points of $E/F$.
%\marginpar{Last sentence was about $F=K(t)$! Might need to reword slightly}
%\marginpar{All lemmas should be stated with specialization at $K_p$-rational places if we want to include ``loc.dim" as currently stated in Theorem 1.1!}
\end{kor}

\begin{lemma}
\label{cyclic_dec_group2}
%\marginpar{Might need some modification!}
Let $G$ be a finite group, let $q,r$ be distinct prime numbers, and let $K$ be a number field 
%not containing the $q$-th roots of unity.
such that $K(\zeta_r)$ does not contain a primitive $q$-th root of unity.
%\marginpar{Assumption has been adapted here, but needs to be adapted in main theorem etc.!}
Let $E_i/F_i$, $i=1,\dots, r$, be finitely many $G$-extensions of transcendence degree $1$ function fields over $K$. Then there exists a non-empty Chebotarev set $\mathcal{S}$ of primes $\mathfrak{p}$ of $K$, completely split in $K(\zeta_r)$, 
such that 
%the inertia group $I$ 
%and the decomposition group $D$ at $\frak{p}$ in a 
 the specialization $(E_i\cdot K_\mathfrak{p})_{t_0}/K_\mathfrak{p}$, at any degree-$1$ place $t_0$ of $F_i\cdot K_\mathfrak{p}$, for any $i=1,\dots, r$, does {\underline{not}} fulfill 
the following condition:\\ 
%all of the following:  %decomposition group $U$ and inertia group $C_r\le U$ at $p$.
\\
\iffalse
 $D$ is abelian,
\marginpar{(Currently trying to remove this ``abelian" assumption from the proof.)}
 $I$ is cyclic of order divisible by $r$ and $D/I$ is divisible by $q$.
 \fi
 %
%
 %$I$ is cyclic, $r\divides \#I$ and $(\#I,q)=1$, and $D/I$ is of order divisible by $q$.
\JK{$(E_i\cdot K_\mathfrak{p})_{t_0}/K_\mathfrak{p}$ is tamely ramified of ramification index divisible by $r$, but coprime to $q$; and of residue degree divisible by $q$.
 }
%for each $p\in \mathcal{S}$ and all degree-$1$ places $t_0$ of $F\cdot K_p$
% \marginpar{Just in case neeeded: Proof of 3.2 shows that ``coprime" assumption can be dropped for abelian $D$.}
 %For the above, the proof should work.
 %
 %
 \iffalse
 cyclic
 inertia group of order $r$ as well as \marginpar{Assumption of cyclic dec.gp. does not seem necessary at this point!} residue degree divisible by $q$ and coprime to $r$ at the prime $p$.
 \fi\end{lemma}

 \iffalse
 {\bf Might be useful (for reduction of ``$E/F$" case to ``$E/K(t)$" case) to word this slightly stronger; e.g., it should be true that ID pair $(C_r, C_{m})$ is impossible for EVERY $m$ divisible by $qr$  (and with the SAME positive density set of primes, independently of $m$.)\\
 Above is a try in that direction; removed the mentioning of the subgroup $U$ because this becomes only necessary when one really wants to conclude about non-parametricity.

 Dropping ``abelian" assumption from the assertion of 3.2 would be sufficient for reduction, but still not sure if possible.
 
 UPDATE: What the proof DOES show is that ``ramification index divisible by $r$ but coprime to $q$, and residue degree divisible by $q$" can be excluded without relying on ``abelian" assumption. Is that enough for reduction? It would be if the rational field $K(t) \subset F (\subset E)$ can be chosen such that $F/K(t)$ has all its ramification coprime to $q$. Since $q$ is not $2$ here, this should be manageable since a ``random" function $t\in F$ should have simple branching?
 }
 \fi
 
 %Want to try this for arbitrary number fields. Condition should only be that cyclotomic 2-extension is cyclic?
\begin{lemma}
\label{cyclic_dec_group3}
Let $G$ be a finite group 
\iffalse containing a cyclic subgroup $U\cong C_{4}$
\fi
and let $K\subset \mathbb{R}$ be a real number field.
Let $E_i/F_i$, $i=1,\dots, r$, be finitely many  $G$-extensions of transcendence degree $1$ function fields over $K$. Then there exists a non-empty Chebotarev set $\mathcal{S}$ of primes $\mathfrak{p}$ of $K$ such that, for any $i\in \{1,\dots, r\}$, 
%for each $p\in \mathcal{S}$ and all degree-$1$ places $t_0$ of $F\cdot K_p$, 
the specialization $(E_i\cdot K_\mathfrak{p})_{t_0}/K_\mathfrak{p}$ at any  degree-$1$ place $t_0$ of $F_i\cdot K_\mathfrak{p}$, does \underline{not} have inertia group of order $2$ \DN{which embeds} into an order $4$ cyclic subgroup of the decomposition group. %of order divisible by $4$ at $\mathfrak{p}$.
\iffalse
{\bf Instead ``target statement": $E_{t_0}$ does not have inertia group $I$ of order $2$ embedding into a cyclic subgroup of order $4$ inside the decomposition group $D$ (at $p$). Proof is already adapted to this statement!}
\fi
\end{lemma}

We divide the proof of Lemmas \ref{cyclic_dec_group2} and \ref{cyclic_dec_group3} into three stages. First, we look at the case of a single extension $E/F$ with rational base field $F=K(t)$.  Next, we derive the case of an arbitrary (single) extension $E/F$. Finally, we progress to finitely many extensions $E_i/F_i$, $i=1,\dots, r$.

\iffalse
\begin{remark}
\label{rem:fin_many}
 The conclusions of Corollary \ref{cyclic_dec_group} and Lemmas \ref{cyclic_dec_group2} and \ref{cyclic_dec_group3} remain valid even if one allows a finite set of $K$-regular $G$-extensions $E_i/F_i$ ($i=1,\dots, r$) %(in which case ``Hilbert-Grunwald property" should be taken to mean that every Grunwald problem outside of a finite set of primes has a solution arising as a specialization of at least one of those extensions)
 instead of just one, with the generalization to be understood such that all primes in the respective set $\mathcal{S}$ fulfill the conclusion for each of the extensions $E_i/F_i$. Indeed, the proofs of all three results use only that the set of all completions at branch points $t_i$ is a finite set, not actually the fact that those completions come from a common function field extension $E/F$.
 \end{remark}
 \fi

%Regarding case ii), we use the following, which, unlike Corollary \ref{cyclic_dec_group} was not directly noticed in \cite{KLN}:
\subsection{Proof of Lemmas \ref{cyclic_dec_group2} and \ref{cyclic_dec_group3}. Part 1: The case $r=1$ and $F_1=K(t)$.}
\label{sec:lemma_pf2}
We \JK{begin} the proof of Lemmas \ref{cyclic_dec_group2} and \ref{cyclic_dec_group3} by showing that they hold in the special case $r=1$ and $F:=F_1=K(t)$. Before going into the proofs of the two individual cases, we note that we may furthermore restrict to specializations at {\it non}-branch points $t_0$ of $E/K(t)$ (up to excluding finitely many more primes from the respective Chebotarev set $\mathcal{S}$), since the finitely many residue fields at branch points can give ramified completions only at finitely many primes; and that moreover, it suffices to consider $K$-rational specializations $t_0\in \mathbb{P}^1(K)$ due to Lemma \ref{lem:pv}.
\begin{proof}[Proof of Lemma \ref{cyclic_dec_group2} \JK{in the case $F=K(t)$}]
Let $\mathcal{S}'$ be the set of primes of $K$ which split completely in $K(\zeta_r)$, but not in $K(\zeta_q)$. Since $K(\zeta_q)$ is not contained in $K(\zeta_r)$, this is clearly a non-empty Chebotarev set. Now let $t_1,\dots,t_r$ be the branch points of $E/K(t)$ and $L/K$ 
be the Galois closure (over $K$) of the compositum of all residue extensions $E(t_i)_{t_i}/K(t_i)$ at these branch points. Let $\mathcal{S}''$ be the set of all primes of $K$ which are unramified in $L(\zeta_{qr})/K$ with residue degree coprime to $q$. We claim that \JK{$\mathcal{S}:=\mathcal{S}'\cap\mathcal{S}''$}  is still a non-empty Chebotarev set. Indeed, $Gal(K(\zeta_{qr})/K) \le G_1\times G_2$, where $1\ne G_1\le C_{q-1}$ and $G_2\le C_{r-1}$, and the kernel of the projection to the second component is non-trivial. The set $\mathcal{S}'$ is the set of all primes whose Frobenius $\sigma$ in $K(\zeta_{qr})/K$ is a non-identity element of this kernel. By Chebotarev's density theorem, to show the claim it suffices to show that any non-identity subgroup $U$ of the cyclic group $G_1=\langle\sigma\rangle$ has a lift to a cyclic subgroup of $Gal(L(\zeta_{qr})/K)$ whose order is still coprime to $q$. But this is elementary, since if $\widehat{U} = \langle\widehat{\sigma}\rangle$ is any such subgroup and $q^e$ the highest $q$-power dividing its order, then $\langle\widehat{\sigma}^{q^e}\rangle$ is a subgroup as desired. 
%We therefore set $\mathcal{S}=\mathcal{S}'\cap \mathcal{S}''$.

Now denote by $\mathcal{R}$ the set of primes $\mathfrak{p}$ of $K$ for which some specialization $E_{t_0}/K$, with $t_0$ a degree-$1$ place of $K(t)$, has cyclic inertia group $I$ of order divisible by $r$ \JK{and coprime to $q$}, %abelian decomposition group $D$,
and residue degree $|D/I|$ divisible by $q$. It suffices to show that $\mathcal{R}\cap \mathcal{S}$ is finite.
Up to exempting a finite set of primes, Proposition \ref{beckmann} yields the following: If a prime $\mathfrak{p}$ of $K$ is ramified of ramification index divisible by $r$ in a specialization $E_{t_0}/K$, then the intersection multiplicity $I_{\mathfrak{p}}(t_0,t_i)$ of $t_0$ and $t_i$ at $\mathfrak{p}$ must be positive for some branch point $t_i$ of $E/K(t)$ whose ramification index is divisible by $r$. We first claim that, for $\mathfrak{p}\in \mathcal{R}\cap \mathcal{S}$ (up to finitely many), this branch point $t_i$ cannot have
ramification index coprime to $q$. Indeed, if it did, then Proposition \ref{mainlemma_kln} would imply that the only way for $\mathfrak{p}$ to have residue degree divisible by $q$ in $E_{t_0}/K$ is that there is a prime $\mathfrak{p}'$ of $K(t_i)$ extending $\mathfrak{p}$ of relative degree $1$ (i.e., $K(t_i)_{\mathfrak{p}'}$ is identified with $K_\mathfrak{p}$), and such that the residue degree at $\mathfrak{p}'$ in $E(t_i)_{t_i}/K(t_i)$ is divisible by $q$. But in this case, the Frobenius of $\mathfrak{p}$ in $L/K$ is certainly of order divisible by $q$, meaning that $\mathfrak{p}\notin\mathcal{S}$. 

Proposition \ref{mainlemma_kln} thus yields that the only way to obtain \DN{a specialization $E_{t_0}/K$ whose decomposition group at $\fp$ is of order divisible by $qr$ and its ramification index is divisible by $r$}   is that the factor $q$ ``stems" from the inertia group at $t\mapsto t_i$; i.e., this inertia group contains subgroups $V_1\le V_2$ such that $|V_1|$ is divisible by $r$ but not by $q$, $|V_2/V_1|$ is divisible by $q$, $V_1$ surjects onto the inertia group $I$ at $\mathfrak{p}$ in $E_{t_0}/K$
 and $V_2/V_1$ injects into $D/I$ (where $D$ is the decomposition group at $\mathfrak{p}$ in $E_{t_0}/K$).

\JK{But also, since the ramification index at $t\mapsto t_i$ is divisible by $q$, it follows that $\zeta_q$ is contained in the residue extension $E(t_i)_{t_i}$, e.g.\ by \cite[Lemma 2.3]{KLN}), and thus by Proposition \ref{mainlemma_kln} also in the completion of $E_{t_0}$ at $\mathfrak{p}$. But $\mathfrak{p}$ is non-split in the  extension $K(\zeta_q)/K$, whence 
 %the decomposition group $D$ at $\mathfrak{p}$ in $E_{t_0}/K$ 
 $D/I$ maps onto a non-trivial subgroup of $\textrm{Gal}(K(\zeta_q)/K)$.}
 
 Finally, \JK{this subgroup acts faithfully on the inertia group at $t\mapsto t_i$ via its action on $\mu_q$, and hence (since $|V_1|$ is coprime to $q$) also on $V_2/V_1$}, 
  meaning that $D/I$ is non-abelian. This is of course a contradiction, since $D/I$ is always cyclic.
%contradicting the requirement that the decomposition group at $p$ should be abelian.

 We have therefore shown that, up to excluding finitely many primes (which can be achieved by choosing a suitable proper Chebotarev subset), the set $\mathcal{S}$ fulfills the assertion of the theorem.
\end{proof}

\begin{proof}[Proof of Lemma \ref{cyclic_dec_group3} in the \JK{case $F=K(t)$}]
Let $\mathcal{S}'$ be the set of primes of $K$ which remain inert in $K(\sqrt{-1})$. 
Let $F/K$ be the Galois closure of the compositum of all residue extensions at branch points of $E/K(t)$. Let $\sigma\in Gal(F/K)$ be a complex conjugation on $F$ and let $\mathcal{S}''$ be the set of primes of $K$ whose Frobenius class in $F/K$ is the class of $\sigma$. It is obvious (upon considering the compositum of $F/K$ and $K(\sqrt{-1})/K$) that $\mathcal{S}'\cap \mathcal{S}''$ is a non-empty Chebotarev set, so we set $\mathcal{S}=\mathcal{S}'\cap \mathcal{S}''$.
Denote by $\mathcal{R}$ the set of primes $\mathfrak{p}$ of $K$ at which some specialization $E_{t_0}/K$ has %decomposition group $U$ and inertia group $C_2\le U$.
inertia group $I = I_{t_0,\mathfrak{p}}$ of order $2$ and decomposition group $D=D_{t_0,\mathfrak{p}}$ containing a cyclic overgroup of $I$ of order $4$. Note that this readily forces $D$ to be cyclic since it centralizes a group $I$ with cyclic quotient.
In analogy to the proof of Lemma \ref{cyclic_dec_group2}, we aim at concluding that $\mathcal{R}\cap \mathcal{S}$ is finite.

Proposition \ref{mainlemma_kln} yields once again that for all but finitely many $\mathfrak{p}\in \mathcal{R}$, there must be a branch point $t_i$ and a prime $\mathfrak{p}'$ extending $\mathfrak{p}$ of relative degree $1$ in $K(t_i)$; furthermore this branch point $t_i$ either has ramification index strictly divisible by $2$ and residue extension $E(t_i)_{t_i}/K(t_i)$ of even local degree at $\mathfrak{p}'$; or it has ramification index divisible by $4$. 

We now claim that furthermore, for $\mathfrak{p}\in \mathcal{R}\cap \mathcal{S}$, each such $t_i$ must be such that the field $K(t_i)$ has a real embedding. Indeed, this is due to the fact that (due to Proposition \ref{beckmann} and up to exempting finitely many $\mathfrak{p}$) for the inertia group at $\mathfrak{p}$ in the specialization $E_{t_0}/K$ to be a non-trivial subgroup of some inertia group at $t_i$ in $E/K(t)$, it is necessary that $t_0$ and $t_i$ meet at $\fp$. This in term implies that the Frobenius of $\mathfrak{p}$ in the Galois closure of $K(t_i)/K$ is an element fixing at least one conjugate of $t_i$. But since for $\mathfrak{p}\in \mathcal{S}$, this Frobenius class is the class of complex conjugation, its elements have a fixed point if and only if $K(t_i)$ has a real embedding.

We may and will therefore assume $K(t_i)\subset \mathbb{R}$ whenever a particular branch point $t_i$ is chosen in the following.

For $t_i$ as above, let $M:=E(t_i)_{t_i}$ be the residue extension at $t_i$, and $L:=M\cap \mathbb{R}$. Note the inclusion $K\subseteq K(t_i)\subseteq L$. We now claim the following:

(*) For $\mathfrak{p}\in \mathcal{R}\cap \mathcal{S}$ (exempting finitely many), there exists at least one $t_i$ as above such that the extension $M/L$ embeds into a $C_4$-extension. 
%Really? Or just M/L^\sigma for some conjugate L^\sigma of L (\sigma \in Gal(M/K(t_i))? (Although then conjugating back gives result for M/L...)

\iffalse
Indeed, for $t_0$ and $t_i$ meeting at $\mathfrak{p}$ as chosen above, Proposition \ref{mainlemma_kln}iii) identifies the decomposition group $D_{t_0,p}$ of $p$ in $E_{t_0}/K$ with a subgroup of the decomposition group $D$ of $t\mapsto t_i$ in $E(t_i)/K(t_i)(t)$ (namely, a subgroup mapping onto $D_{t_i,\mathfrak{p}}$ under the canonical epimorphism $D\to D/I$). Here $D_{t_i,\mathfrak{p}}$ is the decomposition group at $\mathfrak{p}$ in $M/K(t_i)$, i.e., the Galois group of the completion $M_{\mathfrak{p}}/K(t_i)_{\mathfrak{p}}$. 

\marginpar{Is the above all needed?}
\fi

Since this is trivial if $M=L$, we may assume $M = L(\sqrt{a})$(with a non-square $a\in L$) for the moment.
Let $\mathfrak{q}$ be a prime of $L$ extending $\mathfrak{p}'$ of degree $1$ (such a prime exists, due to the fact that the Frobenius at $\mathfrak{p}$ has the same cycle structure as complex conjugation; after all, $M\subset \mathbb{R}$ has a real embedding extending the real embedding of $K(t_i)$). In particular, we may identify $M_{\mathfrak{q}}/L_{\mathfrak{q}}$ with $M_{\mathfrak{p}'}/K(t_i)_{\mathfrak{p}'}$.

Now assume $I_{t_0,\fp}$ and $D_{t_0,\mathfrak{p}}$ are cyclic of order $2$ and divisible by $4$, respectively, which is possible by the choice of $\mathcal{R}$. Since \JK{$t_0$ and $t_i$ are assumed to meet at $\mathfrak{p}'$, in particular they} meet at $\mathfrak{q}$.
Consider thus the specialization $(EL)_{t_0}/L$ of $EL/L(t)$ at $t_0$. Since $M_{\mathfrak{q}}/L_{\mathfrak{q}}$ is quadratic, the residue degree at $\mathfrak{q}$ in this \JK{specialization is even}. Furthermore, the ramification index of $\mathfrak{q}$ in $(EL)_{t_0} =E_{t_0}\cdot L$ is the same as the one of $\mathfrak{p}$ in $E_{t_0}$, since we may without loss of generality assume that $\mathfrak{p}$ is unramified in $L/K$. 
But the decomposition group at $\mathfrak{q}$ in $(EL)_{t_0}/L$ is isomorphic to a subgroup of the one at $\mathfrak{p}$ in $E_{t_0}/K$, and therefore by the above must be cyclic of order \JK{divisible by} $4$. But Proposition \ref{mainlemma_kln}ii) identifies the decomposition group $D_{t_0,\mathfrak{q}}$ of $\mathfrak{q}$ in $(EL)_{t_0}/L$ with a subgroup \JK{$U$} of the decomposition group $D$ at $t\mapsto t_i$ in $EL/L(t)$ (namely, a subgroup mapping onto $D_{t_i,\mathfrak{q}}$ under the canonical epimorphism $D\to D/I$). 
Since the order of $D_{t_i,\mathfrak{q}}$ is in fact the full residue degree of $EL/L$ at $t_i$ (namely, $2$), this yields a subextension $\widehat{EL}_U$ of the completion $\widehat{EL} / L((t-t_i))$ such that $\widehat{EL}/\widehat{EL}_U$ has \JK{cyclic Galois group $U$ of order divisible by $4$}, $\widehat{EL}$ has residue field $M$ and $\widehat{EL}_U$ has residue field $L$. With a suitable parameter $s$ of the Laurent series field $\widehat{EL}_U$, this extension is then of the form $M((\sqrt[d]{bs})) \supset M((s)) \supset L((s))$ for some \JK{even integer $d$ and} $b\in M$. \JK{In particular, $M((s))/L((s)) =L((s))(\sqrt{a})/L((s))$ embeds into a $C_4$-extension, which} is possible if only if $a$ is a sum of two squares in \JK{$L((s))$} (cf.\ \cite[Proposition 16.5.1]{FJ}), and then  (via \JK{comparing the lowest coefficient\footnote{If the two squares happen to have a pole at $s=0$, this comparison yields a representation of $0$ as a sum of two squares in $L$, \DN{at least one of which is nonzero. Such a representation of $0$} is impossible since $L$ is real.}}) even in \JK{$L$}; hence $M/L$ also embeds into a $C_4$-extension.

\iffalse

So, after setting $s:=t-t_i$, the completion $\widehat{EL}/L((s))$ at $s$ must have a $C_4$-subextension factoring through the (at most quadratic) residue field extension $M((s))/L((s))$.  %%Again, issue of subgroup vs quotient group? Strangely no problem in case G abelian...

%Would like to use a prime of L extending \mathfrak{p} of degree 1. But not clear that t_0 and t_i then meet exactly at this prime! 
%(Actually SHOULD be clear from definition, since t_0-t_i\in \mathfrak{p} \subset \mathfrak{q}
%D = Gal(M ((\sqrt[e]{...})) / M((s)) / K(t_i)((s)) )---> D_0 = Gal(M((..))/L
\marginpar{Gap!}
\fi

%(...)
But now note that if $M/L$ were a proper (and then automatically non-real) extension, it could not possibly embed into a $C_4$-extension, or otherwise complex conjugation would have to extend to an order-$4$ automorphism. Hence, $M=L$ is a {\it real} Galois extension of $K(t_i)$. Firstly, since the residue extension $E(t_i)_{t_i}$ at a branch point $t_i$ of ramification index $e$ always contains the $e$-th roots of unity, this already forces $e=2$. But secondly, $M=L$ implies that for any $\mathfrak{p}\in \mathcal{S}$, the above chosen prime $\mathfrak{p}'$ of $K(t_i)$, extending $\mathfrak{p}$, has an extension to $M$ of degree $1$,  due to our choice of Frobenius class at $\mathfrak{p}$ (as in the argument for choosing $\fp'$). However, due to $M/K(t_i)$ being Galois, this already forces $\mathfrak{p}'$ to be completely split in $M$.
%\marginpar{Need to check this! Really split? Or just possesses one degree-$1$ extension?}
%Coset {x,x\tau}; fixed by \tau in coset action iff x^\tau = x => #fixed points ~ #Centralizer, which may be relatively SMALL!!!%
%
%
Therefore, the residue extension at $\mathfrak{p}'$ in $E(t_i)_{t_i}/K(t_i)$ is trivial. Proposition \ref{mainlemma_kln} then implies that for any $t_0\in K$ meeting $t_i$ at $\mathfrak{p}$, the decomposition group at $\mathfrak{p}$ in $E_{t_0}/K$ is a subgroup of the inertia group at $t_i$, and thus of order dividing $e=2$. In particular, this decomposition group cannot contain a cyclic group of order $4$. 
\end{proof}

\subsection{Proof of Lemmas \ref{cyclic_dec_group2} and \ref{cyclic_dec_group3}. Part 2: The case $r=1$ and $F_1$ arbitrary.}
\label{sec:lemma_pf1}
%We begin the proof of Lemmas \ref{cyclic_dec_group2} and \ref{cyclic_dec_group3} by reducing both to the special case $F=K(t)$. Thus, we assume for the moment that both lemmas have been proven for the case of $K$-regular Galois extensions of the form $E/K(t)$. The proof of the latter will be deferred to Section \ref{sec:lemma_pf2}.
\JK{We now extend the proof of Lemmas \ref{cyclic_dec_group2} and \ref{cyclic_dec_group3} to a single extension $E/F$ with arbitrary, not necessarily rational base fields $F$.} 
We will make use of the following technical lemma. \JK{For our purposes, we will only need conclusion i), but the proof easily  yields ii) as well, which may be of interest by itself.}
\begin{lemma}
\label{techn_lemma}
Let $K$ be field of characteristic $0$ \JK{with algebraic closure $\overline{K}$}, and $F$ a (one-variable) function field \JK{over} $K$. Let $\mathcal{R}$ be a finite set of places of $F\cdot \overline{K}$.
%\marginpar{Word with ``points" instead? Then regularity could already be dropped here.}
 Then there exists a non-constant function $\tau\in F$ such that both of the following hold:
\begin{itemize}
    \item[i)] The \JK{places} in $\mathcal{R}$ lie over pairwise distinct \JK{places} of \JK{$\overline{K}(\tau)$}, all of which are \JK{unramified in $F\cdot \overline{K}$}.
    \item[ii)] $F/K(\tau)$ is ``simply branched", i.e., every place \DN{of $\overline{K}(\tau)$ which ramifies in $F\cdot \overline{K}$ has a unique ramified place over it, and the latter has ramification index $2$.}
    %is extended by only a single \JK{place of} ramification index $2$, and \JK{places of} of ramification index $1$ otherwise.
\end{itemize}
\end{lemma}
\begin{proof}
\DN{Below we use standard facts from algebraic geometry, general references for which are \cite[\S 4.3]{Har},\cite[\S 1.3]{Rus}}.
\JK{Up to replacing $K$ by its algebraic closure in $F$, we may choose a smooth projective, geometrically irreducible curve $C$ over $K$} with function field $F$. Via a suitable projection $\varphi:C\to \tilde{C}\subseteq \mathbb{P}^2_{\overline{K}}$, we may choose a plane curve model of $C$ having at most double points as singularities, and such that all points $p\in C$ belonging to places in $\mathcal{R}$ map to non-singular points $\varphi(p)$ of $\tilde{C}$. If $\tilde{C}$ is the projective curve $f(X,Y,Z)=0$, we may assume $F\cong K(x,y)$, where $f(x,y,1)=0$. Consider now the projection $\pi:\tilde{C} \to \mathbb{P}^1$, $(x:y:z)\mapsto (y+\epsilon x:z)$ with $\epsilon \in K$. Setting $\tau=\JK{\tau(\epsilon):=}y+\epsilon x$, this yields an extension $F/K(\tau)$. Now note that each of the following is a finite set: \\
\noindent 1) The set of tangents to the curve $\tilde{C}$ passing through some point $\varphi(p)$ (with $p$ corresponding to a place in $\mathcal{R}$, as above); 2) the set of lines \JK{in $\mathbb{P}^2$} passing through some point $\varphi(p)$ as well as some singular point of $\tilde{C}$; 3) the set of lines passing through more than one $\varphi(p)$; 4) the set of ``bitangents" to $\varphi(C)$ (i.e., tangents at two different points); 5) the set of tangents to $\tilde{C}$ passing through a singular point of $\tilde{C}$; and 6) the set of lines intersecting $\tilde{C}$ with multiplicity $\ge 3$  \JK{at some point} (indeed, we have assumed that  $\tilde{C}$ has at most double points, whence the latter lines are only the tangent lines at a singularity or a flex). 

Therefore, up to choosing $\epsilon\in K$ appropriately, we may assume that none of the lines $Y+\epsilon X = aZ$, $a\in K$, \DN{in $\mathbb P^2$} are of the above form. But note that the intersection of any such line with $\tilde{C}$ is the fiber $\pi^{-1}(a)$.\footnote{We have ignored the line at infinity for convenience, but note that at any rate a suitable prior transformation on $\tau$ ensures that $\pi$ is unramified at $\tau\mapsto \infty$.}
The above 1)-6)  thus ensure that all $\varphi(p)$ lie in distinct fibers, each of these fibers has maximal cardinality $[F:K(\tau)]$; and any other fiber has (at most one double point, hence) cardinality at least $[F:K(\tau)]-1$. The former observation implies i), while the latter implies ii).
\end{proof}

Assume that $E/F$ is a $G$-extension of transcendence degree $1$-function fields over $K$. 
%, Let $\kappa$ be the algebraic closure of $K$ in $E$, so that $E/(F\cdot \kappa)$ is $\kappa$-regular, and l
Let $\mathcal{R}$ be the set of \JK{places of $F\cdot \overline{K}$} ramifying in $E\cdot \overline{K}$.

By Lemma \ref{techn_lemma}, there exists a non-constant function $\tau\in F$ such that the \JK{places} in $\mathcal{R}$ extend pairwise distinct \JK{points $\tau\mapsto \tau_0\in \mathbb{P}^1(\overline{K})$}, and the latter are all unramified in $F/K(\tau)$.

\JK{Consider now the Galois closure $\Omega/K(\tau)$ of $E/K(\tau)$, and let $\Gamma$ be its Galois group.}
\iffalse
Consider now the Galois group $A$ of the Galois closure $\Omega/K(\tau)$ of $E/K(\tau)$, as well as $\Gamma:=\Gal(\Omega \overline{K}/\overline{K}(\tau))$. Then $A$ embeds naturally into the imprimitive wreath product $G\wr S_d =G^d\rtimes S_d$, $d:=[F:K(\tau)]$, via the inclusion $E\supseteq F\supseteq K(\tau)$. Furthermore, $\Gamma$ is generated by the inertia group generators at ramified places of $\Omega/K(\tau)$. 
Since, a ``simple" branch point has an inertia group generated by a transposition and since a transitive subgroup of $S_d$ generated by transpositions is necessarily $S_d$, Condition ii) of Lemma \ref{techn_lemma} implies that $\Gamma$ projects onto $S_d$. Furthermore, by Condition i) of Lemma \ref{techn_lemma}, the inertia groups at branch points ramified in $\Omega/K(\tau)$, but not in $F/K(\tau)$, are all generated by elements of the block kernel $G^d \cap A$ which are non-trivial on  a single of the $d$ components. By transitivity of the action on the components, we may conjugate these elements to non-trivial elements on the first component, and deduce that $G_{(1)} \cong G$. Moreover, transitivity implies that $\Gamma$ actually contains all of $G^d$, and since we already know that it projects onto $S_d$, we have $\Gamma=G\wr S_d =A$. In particular, the equality $\Gamma=A$ shows that $\Omega/K(\tau)$ is $K$-regular.

\marginpar{$K$-regularity should be replaced /previous paragraph obsolete.}
\fi
Let $t_0$ be a $K_\mathfrak{p}$-rational place of $F\cdot K_\mathfrak{p}$ and $\tau_0$ the underlying place of $K_\mathfrak{p}(\tau)$. Applying Proposition \ref{beckmann} to $\Omega/K(\tau)$ and to the Galois closure of $F/K(\tau)$, we obtain the following conclusion for all but finitely many primes $\mathfrak{p}$ of $K$, with the exceptional set depending on $\Omega$ but not on the particular choice of specialization point $\tau_0$. 

\JK{If $\mathfrak{p}$ ramifies in the specialization of $EK_\fp/FK_\fp$ at $t_0$, then the inertia group generator is a power of an inertia group generator at some branch point of $E/F$. By condition i) of Lemma \ref{techn_lemma}, the ramification index at this branch point is in fact the same as the ramification index at the underlying branch point of $\Omega/K(\tau)$; consequently, the ramification indices at $\mathfrak{p}$ in $(EK_\fp)_{t_0}/(FK_\fp)_{t_0}$ and in $(\Omega K_\fp)_{\tau_0}/K_\fp$ are identical.}

\iffalse
If $\mathfrak{p}$ ramifies in the specialization $(\Omega K_\mathfrak{p})_{\tau_0}/K_\mathfrak{p}$, then one of the following two conditions holds:\\
1) the ramification index in this specialization is the same as in the specialization $(EK_\mathfrak{p})_{t_0}/(FK_\mathfrak{p})_{t_0}$ of $EK_\mathfrak{p}/FK_\mathfrak{p}$ at $t_0$.
This happens when $\tau_0$ meets, at $\mathfrak{p}$, a point $\tau_i \in \mathbb{P}^1(\overline{K})$ underlying a ramified place of $E/F$. \\
2) $(EK_\mathfrak{p})_{t_0}/(FK_\mathfrak{p})_{t_0}$ is unramified at $\mathfrak{p}$. This happens when $\tau_0$ meets, at $\mathfrak{p}$, 
%a point $t_i \in \mathbb{P}^1(\overline{K})$ underlying a ramified place of $E/F$, whereas the second happens when $\tau_0$ meets 
a branch point of $F/K(\tau)$. \\
\fi
Furthermore the residue degree at $\mathfrak{p}$ in $(\Omega K_\mathfrak{p})_{\tau_0}/K_\mathfrak{p}$ is trivially a multiple of the residue degree in $(EK_\mathfrak{p})_{t_0}/(FK_\mathfrak{p})_{t_0}$. Now, for each of Lemma \ref{cyclic_dec_group2} and \ref{cyclic_dec_group3}, consider the set $\mathcal{S}$ of exceptional primes corresponding to the \JK{$\Gamma$-}extension $\Omega/K(\tau)$. It then follows immediately from the above that this set $\mathcal{S}$ also fulfills the respective assertion for the \JK{$G$}-extension $E/F$, so that the assertion follows from the case of a rational base field, treated in Section \ref{sec:lemma_pf2}.
%reduction to the rational function field case \DN{$F=K(\tau)$}.

\subsection{Proof of Lemmas \ref{cyclic_dec_group2} and \ref{cyclic_dec_group3}: General case} 
\DN{Finally we extend the proof of Lemmas \ref{cyclic_dec_group2} and \ref{cyclic_dec_group3} from one to finitely many extensions $E_i/F_i$}. This  only requires to note that, in the proofs in Section \ref{sec:lemma_pf2}, we used only that the set of all completions at branch points $t_i$ of $E/K(t)$ is a finite set, not actually the fact that those completions come from a common function field extension $E/K(t)$. This immediately yields the generalization to finitely many fields of the form $E_i/K(t)$, $i=1,\dots,r$; To get the conclusion for the most general case of finitely many arbitrary $G$-extensions $E_i/F_i$, we perform the reduction argument of Section \ref{sec:lemma_pf1} to each of them.

\section{Proof of the main results}
\label{sec:proofmain}
Here, we will prove our Main Theorem\JK{s \ref{thm:hgmain} and \ref{thm:qi}. In fact, we will prove the following stronger version of Theorem \ref{thm:hgmain}.
\begin{theorem}
\label{thm:hgmain_strong}
Let $G\ne \{1\}$ be a nontrivial finite group, and let $K\subset \mathbb{R}$ be a real number field such that the cyclotomic extensions $K(\zeta_p)$, with $p$ running through the prime numbers, are pairwise distinct. Then the following are equivalent:
\begin{itemize}
\item[i)] $\textrm{ld}_K(G)=1$.
\item[ii)] $\textrm{hgd}_K(G) = 1$.
\item[iii)] There exists a $K$-regular Galois extension $E/K(t)$ with group $G$ possessing the Hilbert--Grunwald property.
\item[iv)] $G$ is either a cyclic group of order $2$ or an odd prime power; or $G$ is a Frobenius group whose kernel and complement both are cyclic groups of order $2$ or an odd prime power.
\end{itemize}
\end{theorem}

Note that $\textrm{ld}_K(G)\le 2$ was shown in \cite[Main Theorem]{KN} for all groups $G$ and number fields $K$, whence Theorem \ref{thm:hgmain_strong} completely determines the local dimension over fields $K$ as above. 
While the implications iii)$\Rightarrow$ii)$\Rightarrow$i) are trivial from the definitions, the remaining implications i)$\Rightarrow$iv) and iv)$\Rightarrow$iii) will be shown below} as Theorems \ref{thm:hgp} and \ref{thm:cyclic_odd}. % (whose combination is exactly Theorem \ref{thm:hgmain}).

Firstly, we use the lemmas of the previous section to obtain a \JK{necessary} condition for a group to be of local dimension $1$ over   %$\textrm{ld}_K(G)=1$ for a group $G$. 
 certain real number fields $K\subset \mathbb{R}$. 
 %over which the lemmas apply}. 
 %, in order to make sure that the assumptions of Lemmas \ref{cyclic_dec_group2} and \ref{cyclic_dec_group3} are fulfilled without further dependence on the concrete group $G$.
\begin{theorem}
\label{thm:hgp}
Let $K\subset \mathbb{R}$ be a real number field for which the extensions $K(\zeta_p)$ (with $p$ prime) are pairwise distinct, and let $G$ be a finite group such that $\textrm{ld}_K(G)=1$.
\iffalse
Assume that there exists a $K$-regular $G$-extension $E/K(t)$ which has the Hilbert-Grunwald property.

\marginpar{Due to rewording of Main Theorem, the assumption here should be changed to  ``Assume $\ld_K(G)=1$". (Proof unchanged).}
\fi
Then $G \cong C_{P}\rtimes C_Q$, where $P$ and $Q$ are coprime, each either $\le 2$ or an odd prime power, and $C_Q$ acts on $C_P$ as a subgroup of $\Aut(C_P)$.\footnote{In particular, for each order there is at most one group with this property, due to $\Aut(C_P)$ being cyclic.} 
%
% Surely, those would have generic extensions. Does that enable a generalization of the construction for cyclic p-groups below??
%
%containing either a noncyclic abelian subgroup, or an element whose order is either $4$ or not a prime power. Then no $K$-regular extension $E/K(t)$ has the Hilbert-Grunwald property.
\end{theorem}
\begin{proof}
Due to Corollary \ref{cyclic_dec_group}, $G$ cannot have a non-cyclic abelian subgroup. Indeed, if it did, then it would also contain a subgroup $C_q\times C_q$ for some prime $q$. Letting $L$ be  as in Corollary \ref{cyclic_dec_group},   the complete field $K_\fp$ possesses a $(C_q\times C_q)$-extension for all primes $\fp$ of $K$ splitting completely in $K(\zeta_q)$, whereas such an extension cannot occur as a $K_\fp$-specialization of the prescribed finitely many $G$-extensions $E_i/F_i$, $i=1,\dots, r$, for all primes $\fp$ split in $L(\zeta_q)$. This point was already noted in \cite[Theorem 6.3]{KLN}, and forces all Sylow subgroups of $G$ to be either cyclic or a generalized quaternion group (see \cite[Chapter XII, Theorem 11.6]{CE}).

 But $G$ also cannot have an element of order $4$ due to Lemma \ref{cyclic_dec_group3}. Indeed, every field $K_\fp$ has a $C_4$-extension of ramification index $2$, whence Lemma \ref{cyclic_dec_group3} yields an obstruction (over real fields $K$) for all groups $G$ containing a cyclic subgroup of order $4$, implying in particular that all Sylow subgroups of $G$ must be cyclic. Then it is known that $G=B.A$ is an extension of a cyclic group $A$ by a cyclic group $B$ of order coprime to that of $A$  (see, e.g., \cite[Chapter V, Theorem 11]{Zass}). 
 Finally, all elements of $G$ have prime power order. Indeed, the set $\mathcal{S}$ in Lemma \ref{cyclic_dec_group2} is chosen such that  $K_\fp$ has a totally ramified $C_r$-extension for every $\fp\in \mathcal{S}$ (and of course also an unramified $C_q$-extension), yielding a $C_{qr}$-extension which is not reached via specialization of the prescribed  finitely many extensions $E_i/F_i$, $i=1,\dots, r$. Therefore $B$ and $A$ above must be cyclic of order $2$ or odd prime power order. %If the orders of $B$ and $A$ were powers of the same prime, then $G$ (being Sylow-cyclic) would already be cyclic. 
 Since $|A|$ and $|B|$ are coprime, $G=B\rtimes A$, and the induced homomorphism $A\to \Aut(B)$ must be injective since otherwise there would be non-trivial subgroups of $A$ and $B$ commuting, yielding an element whose order is not a prime-power. This concludes the proof.
 \end{proof}

 \begin{remark}
\label{rem:hilbertgrunwald}
It was shown in \cite{DG} that the set of specializations of a single $K$-regular $G$-extension $E/K(t)$ {\it always} provides a positive answer to all {\it unramified} Grunwald problems (outside some finite set $S_0$) - i.e., where $L^{(\fp)}/K_\fp$ is unramified for all $\fp\in S$. Theorem \ref{thm:hgp} and the preceding lemmas show in particular that this is no longer true as one passes from unramified to ramified Grunwald problems even when the ramified local extensions are chosen cyclic. % for Grunwald problems involving non-cyclic local extensions (situation of Corollary \ref{cyclic_dec_group}) this was observed already in \cite{KLN}; however, the above shows the failure of this Hilbert-Grunwald property even in the cyclic case, under some mild assumptions on $G$ and $K$.
\end{remark}

We note that a similar proof gives the forward direction of Theorem \ref{thm:qi}:
\begin{proof}[Proof of Theorem \ref{thm:qi}, forward direction]
As in the proof of Theorem 
\ref{thm:hgp}, Corollary \ref{cyclic_dec_group} implies that the Sylow subgroups of $G$ are either cyclic or a generalized quaternion group, and Lemma \ref{cyclic_dec_group2} implies that all elements have prime power order. 

If the $2$-Sylow subgroups of $G$ are cyclic, then all Sylow subgroups are cyclic, and $G=B.A$ is an extension of a cyclic group $B$ by a cyclic group $A$ of order coprime to that of $B$ by \cite[Chapter V, Theorem 11]{Zass}. Furthermore, as all elements are of prime power order, $A$ and $B$ are of prime power order and the action of $B$ on $A$ is faithful, so that $B$ embeds into $\Aut(A)$. 

Assume the $2$-Sylow subgroups of $G$ are generalized quaternion groups. 
\JK{Then, as a consequence of the Suzuki-Zassenhaus theorem (see \cite{AM}, Theorem 6.15 together with Remark 6.16) $G$ has a unique, and hence central, involution $z$. By multiplying by $z$, one obtains an element of non-prime-power order, unless $G$ itself is a $2$-group. In the latter case, $G$ itself is therefore a generalized quaternion group.} 
%If $G$ is nonsolvable, then    \cite[Theorem C]{Suz55} implies $G$ contains $\SL_2(\mathbb F_p)$ for odd $p$ and hence has an element of order $2p$ which is not a prime power order.  Hencce $G$ must be solvable. A solvable group with cyclic odd Sylow subgroups and generalized quaternions $2$-Sylow subgroups must be ** 
%by Zassenhauss (??). 
\end{proof}

We now aim at the converse of Theorem \ref{thm:hgp} by showing that groups of the form $G\cong C_P\rtimes C_Q$ as above do in fact admit $K$-regular $G$-extensions with the Hilbert--Grunwald property.
\begin{theorem}
\label{thm:cyclic_odd}
Let $K$ be a number field, and $G =C_{P}\rtimes C_Q$, where $P$ and $Q$ are coprime, each either $\le 2$ or an odd prime power, and $C_Q$ acts on $C_P$ \DN{faithfully}.
%as a subgroup of $\Aut(C_P)$.
%
%Let $p$ be a prime number and $d\in \mathbb{N}$. Let $K$ be a number field such that $K(\zeta_{p^d})/K$ is cyclic,\footnote{In particular, this is always fulfilled if $q$ is odd or $\sqrt{-1}\in K$.} and let $G$ be the cyclic group of order $p^d$.
Then there exists a $K$-regular $G$-extension ${E}/K(t)$ with 
%positive local specialization density.
the Hilbert--Grunwald property.
\end{theorem}

The following lemma constructs  cyclic extensions of certain complete fields, which will then serve as completions (at suitable branch points) of extensions $E/K(t)$ in the proof of Theorem \ref{thm:cyclic_odd}.    Let $v_2:\mathbb Q^\times\ra \mathbb Z$ denote the normalized $2$-adic valuation. 
%We say for short that an extension of a complete field $F$ is an $(e,d)$-extension if it is of degree $d$ and has ramification index $e$. 
%We now make use of the following lemma:
\begin{lemma}\label{lem:construct1}
Suppose $e\,|\,d$ are positive integers with the same prime divisors and either  $v_2(d)=v_2(e)\leq 1$ or $v_2(e)\geq 2$. Suppose further that $\mu_d\cap K(\zeta_e) = \mu_e$, and let
 $E=E^{(e,d)}$ be the splitting field of  $P(X) = X^{d} - \zeta_{e} y^{d/e} \in F[X]$ over  $F=F^{(e,d)}:=K(\zeta_{e})((y))$. Then $E/F$ is a $C_d$-extension with ramification index $e$ and residue extension $K(\zeta_{d})/K(\zeta_{e})$.
\end{lemma}
\begin{proof}
%We first assume 
Let $x= (\zeta_{e} y^{d/e})^{1/d}$ be a root of $P$. 
Then $x^{e} = \zeta_{e}^{e/d}y$, 
%where $\zeta_{p^f}$ is a primitive $p^f$-th root of unity, 
so that $F(x)$ contains $K(\zeta_{d})((y))$.  
Furthermore, $x$ has $y$-adic valuation $1/e$, so that the extension $F(x)/K(\zeta_{d})((y))$ is of degree divisible by $e$.
Since in addition $e$ and $d$ have the same prime divisors and $\mu_d\cap K(\zeta_e)=\mu_e$, the degree of $K(\zeta_{d})/K(\zeta_e)$ is $d/e$, and hence:
$$
%\begin{array}{cc}
[F(x) : F]  = [F(x) : K(\zeta_{d})((y))] \cdot [K(\zeta_{d})((y)): F] 
 \ge e\cdot d/e=d, 
%\end{array}
$$ 
%where the latter is due to the assumptions on $e$. 
This forces equality, and in particular shows the irreducibility of $P$. 
On the other hand, $F(x) = K(\zeta_{d})((y))(x)$ is already a splitting field of $P$. The claims about inertia subgroup and residue extension are obvious from the above treatment.

It remains to show that $E/F$ is cyclic. For a prime $p\divides d$, let $e_p$ and $d_p$ denote the maximal $p$-powers dividing $e$ and $d$, respectively. Note that $E$ contains a field $E_p=F(x_p)$, where $x_p^{d_p}=\zeta_{e_p}y^{d_p/e_p}$ for some primitive $e_p$-th root of unity $\zeta_{e_p}$. By the above,
%$[E_p:F]=d_p$,
\JK{$E_p/F_p$ is Galois of degree $d_p$,} and hence $E$ is the compositum of the fields $E_p$, where $p$ runs over prime divisors of $d$. It therefore suffices to show that $E_p/F$ is cyclic for each prime $p\divides d$. As we have reduced  the claim to this case, we shall henceforth assume $d$ is a power of a single prime $p$. 
%If $p$ is odd, then the Galois group of $P$ is cyclic as a transitive subgroup of $C_d\rtimes \Aut(C_d)$.  

By assumption, if $p=2$ either $v_2(e)=v_2(d)=1$ or $v_2(e)\geq 2$. In the former case, $E/F$ is quadratic and hence cyclic. Henceforth assume either $p$ is odd or that $p=2$ and $4\divides e$ in which case $\sqrt{-1}\in F$. It follows that  $F(\zeta_e^{1/d})/F$ has cyclic Galois group   $\langle\tau\rangle$ of order $d$. As $\zeta_e\in F$, the extension $F(y^{1/e})/F$  has cyclic Galois group $\langle \sigma\rangle$ of order $e$. As the first is unramified and the second is \JK{totally} ramified, these two extensions are linearly disjoint over $F$, and hence $F(\zeta_e^{1/d},y^{1/e})/F$ is Galois and its Galois group may be identified with $\langle\sigma\rangle\times\langle\tau\rangle$, where $\sigma$ fixes $\zeta_e^{1/d}$ and $\tau$ fixes $y^{1/e}$. Noting that $\tau^{d/e}$ fixes $\zeta_d=(\zeta_{e}^{1/d})^e$, observe that   $\tau^{d/e}(\zeta_e^{1/d})/\zeta_e^{1/d}$ and $\sigma(y^{1/e})/y^{1/e}$ 
are primitive $e$-th roots of unity. We may replace $\tau$ by a coprime to $p$ power of $\tau$ to assume these two roots of unity are equal. In particular, $\sigma\tau^{-d/e}$ fixes $x=\zeta_e^{1/d}y^{1/e}$. Since $\sigma\tau^{-d/e}$ is of order $e$, 
the degree of $F(\zeta_e^{1/d},y^{1/e})^{\sigma\tau^{-d/e}}/F$ is $de/e=d=[F(x):F]$. 
%Since $F(x)/F$ is also of degree $d$,   %$F(\zeta_e^{1/d},y^{1/e})^{\sigma\tau^{-d/e}}$ are of degree $d$ (since $\sigma\tau^{-d/e}$ is of order $e$), 
%it follows that $F(x)/F$ is the extension 
It follows that $F(x)/F$ is the extension fixed by  $\sigma\tau^{-d/e}$, and hence its Galois group is isomorphic to $\langle \sigma,\tau\rangle/\langle \sigma\tau^{-d/e}\rangle$. The latter is cyclic, generated e.g.\ by the coset of $\tau$, proving the claim.
%denote It suffices to show that for each prime $p$ dividing
%If $v_2(e)=v_2(d)\leq 1$, the Galois group of $P$ is then an order $d$ transitive subgroup of $C_{d}\rtimes \Aut(C_{d})$, and thus necessarily equal to $C_{d}$. 
%If $s:=v_2(e)$ is at least $2$, then the Galois group of $P$ is an order $d$ transitive subgroup of $C_{d}\rtimes \Aut(C_{d})$  whose projection to $C_{2^s}\rtimes \Aut(C_{2^s})$ is cyclic and does not contain the inversion automorphism, and hence equals $C_{2^s}$.   
\end{proof}

Recall that the only primes $\fp$ for which $K_\fp$ can  have a  cyclic extension of ramification index $e'$ are those for which $\zeta_{e'}\subseteq K_\fp$, see \S\ref{sec:local}. 
%is partitioned as the union of sets We consider specializations over $K_\fp$ at the set 
We divide such primes into the sets $\mathcal S_{e}^{(d)}$ of all primes $\fp$ such that $\mu_d\cap K_\fp = \mu_{e}$, where $e$ runs over positive integers which divide $d$ and are divisible by $e'$. 
%ivfor every $e'\divides d$ and $e\divides e'$. 
%for every $d$ that is divisible by $e$ and has the same prime divisors as $e$. 
%Note that 
%every prime $\fp$ with $\mu_e\subseteq K_\fp$ is in $\mathcal S(e',d)$ for some $e'$ dividing $d$, so that 
%the union over all $\mathcal S_{e'}$, $e\divides e'$, consists of all primes $\fp$ for which $\mu_e\subseteq K_\fp$. 
For each set $\mathcal S_e^{(d)}$, we consider specializations over $K_\fp$ for all  $\fp\in \mathcal S_e^{(d)}$. Letting $E^{(e,d)}$ be the extension from Lemma \ref{lem:construct1}, we have:
%which split completely in $K(\zeta_{e})$ are the only ones . 

%We therefore obtain condition (*) above for the fixed pair $(e,f)$

\begin{lemma}\label{lem:spec-odd}
Let $e\divides d$ be positive integers with the same prime divisors such that $\mu_d\cap K(\zeta_e)=\mu_e$, and either $v_2(e)=v_2(d)\leq 1$ or $v_2(e)\geq 2$. 
Let $E/K(t)$ be an extension with completion $E^{(e,d)}/K(\zeta_{e})((y))$. 
Then for all but finitely many primes $\fp\in \mathcal S_e^{(d)}$,  
%of $K$ such that $\mu_d\cap K_\fp = \mu_e$,
%with  $e=\gcd(d,N(\fp)-1)$,  %and $e$ is odd (resp.\ $e$ is even),
every  $C_d$-extension of $K_\fp$ with ramification index dividing $e$ is a specialization of $EK_\fp/K_\fp(t)$. 
%\item if $e$ is even, 
%\end{enumerate}
\end{lemma}
\begin{proof}
%Fix $e'|e$. 
Using Proposition \ref{mainlemma_kln}iii) and the fact that $\fp$ splits completely in $K(\zeta_{e})$, we find (infinitely many\footnote{Indeed, every $t_0$ such that $t_0-\zeta_{p^e} (\in K_\fp!)$ is of $\fp$-adic valuation $1$ is good enough.}) $t_0\in K$ such that the completion of $E_{t_0}/K$ at $\fp$ is ramified of index $e$ and whose residue degree at $\fp$ is the residue degree of $\fp$ in $K(\zeta_{d})/K(\zeta_{e})$. Since $\fp\in\mathcal S_e^{(d)}$, this degree is $d/e$, and hence $(E_{t_0})_\fp/K_\fp$ is of degree $d$. %p^{f-e}$. 
Since $E^{(e,d)}/F$ has Galois group $C_d$, it follows that  $\Gal((E_{t_0})_\fp/K_\fp)$ is the entire group  $C_d$. 
%-extension.
%n $(e,d)$-extension. 

Using now additionally \JK{Proposition \ref{thm:KLN2}}, we obtain that in fact {\it all} $C_d$-extensions of ramification index  $e'\,|\, e$ are in $Sp(E\cdot K_\fp/K_\fp(t))$. This holds for all primes $\fp\in \mathcal S_e^{(d)}$ as these split completely in $K(\zeta_{e})$ and remain inert in $K(\zeta_{d})/K(\zeta_{e})$. %But note that for any prime which splits completely in a proper extension of $K(\zeta_{p^e})$ the result for the pair $(e,f)$
%follows from carrying out the above argument with $(\tilde{e},f)$ for a suitable $\tilde{e}>e$! 
\end{proof}

\iffalse
Consider now the extension $K(\zeta_{p^f})(( (\zeta_{p^f} t)^{p^{-e}})) / K(\zeta_{p^e})((t))$. 
%
\marginpar{Verify this claim! Should at least give an argument why below polynomial is irreducible!}
%
This is a cyclic extension of complete fields of degree $p^f$ with residue extension $K(\zeta_{p^f})/K(\zeta_{p^e})$ (indeed, the degree $p^f = p^{e} \cdot p^{f-e}$ and residue extension are clear, whereas the claim that the extension is Galois with cyclic Galois group follows, e.g., upon noting that it is a root field (and then for degree reasons already the splitting field) of the polynomial $X^{p^f} - \zeta_{p^e} t^{p^{f-e}}$).
\fi
%
% Capelli Lemma? f:=X^p-\zeta_{p^{e+f-1}) t^{p^{1-e}), g:=X^{p^{f-1}}-\zeta_{p^e} t^{p^{f-e}}  => F=f(g(X)).
% Inductively, g is irreducible. Is f irreducible over root field of g? True if and only if $\zeta_{p^{e+f-1}) t^{p^{1-e})$ is not a p-th power in 
% $K(\zeta_{p^{f-1}})(( (\zeta_{p^{f-1}} t)^{p^{1-e}}))
%
% Useful to rewrite $K(\zeta_{p^f})(( (\zeta_{p^f} t)^{p^{-e}})) = ...?$
%

\begin{proof}[Proof of Theorem \ref{thm:cyclic_odd}]
%Let $p$ be the prime divisor of $P$ and $q$ the prime divisor of $Q$ (in case $Q=1$, $q$ is simply not needed below).
Following Lemma \ref{lem:pv}, we may treat primes $\fp$ of $K$ separately in what follows. In the following the prime $\fp$ is always tacitly assumed to be outside some finite set of exceptional primes arising from  Proposition \ref{mainlemma_kln}.
%Due to Proposition \ref{mainlemma_kln}, 
Due to \cite{DG}, for any $K$-regular $G$-extension ${E}/K(t)$ and all but finitely many primes $\fp$ of $K$, all {\it unramified} extensions of $K_\fp$ with group embedding into $G$ are in $Sp({E}\cdot K_\fp/K_\fp(t))$. We may therefore focus on (tamely) ramified extensions of $K_\fp$. 

Note that
our group $G$ is chosen such that the only cyclic subgroups are subgroups of $C_P$ or (up to isomorphism) of $C_Q$. Thus, 
%the Galois group of a completion of a tame a Galois extension is either of prime power order, or 
every subgroup of $G$ is of the form $U\rtimes V$ with non-trivial subgroups $U\le C_P$ and $V\le C_Q$ such that $V$ acts as a subgroup of  $\Aut(U)$. 
%(it may be worth noting here that the assumptions imply $Q | p-1$, i.e., $Q$ embeds into the automorphism group of every non-trivial subgroup of $P$).
% Prime power order case is dealt with completely below!
%
% Regarding U\rtimes V (say, orders (p^e, q^f)) case: primes in question are exactly those v with $[K_v(\zeta_{p^e}) : K_v] = q^f$
% For those, let $K\subset F\subset K(\zeta_{p^e})$ be the unique subfield of index $q^f$. Then actually local extension $K(\zeta_{p^e}) (( (t-t_i)^{p^{-e}}))((/F((t-t_i))$ is good enough! (splitting field of X^{p^e} - (t-t_i) )
%
It then suffices to construct an extension ${E}/K(t)$ fulfilling the following:
\begin{itemize}
\item[(*)] For each pair $(e,d)$ such that $e>1$, $e\divides d$, and $d\divides P$, and for all but finitely many primes $\fp$ of $K$ such that $K_\fp$ possesses a $C_{d}$-extension of ramification $e$,  
%(in the following for short: $(p^e,p^d)$-extension), a
all $C_{d}$-extensions of $K_\fp$ with ramification index $e$ are in $Sp({E}\cdot K_\fp/K_\fp(t))$.
\item[(**)] The same with $P$ replaced by $Q$.
\item[(***)] For every pair $(e,f)$ such that $e,f>1$,  $e\divides P$, and $f\divides Q$,  and for all but finitely many primes $\fp$ of $K$ such that $K_\fp$ possesses a $(C_e\rtimes C_f)$-extension of ramification index $e$ 
%with inertia group $C_{e}$ and decomposition group $C_{P}\rtimes C_{Q}$ 
(with the semidirect product inherited from $C_P\rtimes C_Q$), each such extension 
%(in the following for short: $e\rtimes d$-extension) 
is in $Sp({E}\cdot K_\fp/K_\fp(t))$.
\end{itemize}

We start with Condition (*). Since $P$ is either an odd prime power or $P=2$, the numbers  $e$ and $d$ have the same prime divisors for each pair $\pi=(e,d)$ as in (*) and $v_2(e)=v_2(d)\leq 1$. For each such pair $\pi$ such that $\mu_d\cap K(\zeta_e)=\mu_e$, we fix a $K(\zeta_e)$-rational place $t\ra t_\pi$ so that the completion at $t\ra t_\pi$ is isomorphic to the field $F^{(\pi)}=F^{(e,d)}$
%= K(\zeta_e)((y))$ 
from Lemma \ref{lem:construct1}. Moreover, we choose all places $t\mapsto t_\pi$ to be distinct.  %  for the time being, fix a pair $(e,d)$ as above.}
%\JK{The proof previously began by additionally requiring $v$ to fulfill $\mu_{p^f}\cap K_v = \mu_{p^e}$. Surely this assumption cannot just be dropped?? (See below suggestion)}
%A decisive observation for the following construction is now that the group 
Since $C_P\rtimes C_Q$ possesses a generic polynomial over $K$  \cite[Theorem 7.2.2]{JLY}, 
 a theorem of Saltman \cite[Theorem 5.9]{S82} implies there
%this implies the solvability of all Grunwald problems for the group $G$ over the function field $K(t)$. 
%Then by Saltman's theorem, 
 exists a $(C_P\rtimes C_Q)$-extension $E/K(t)$ whose completion at each of the places $t\ra t_\pi$ is the extension $E^{(\pi)}/F^{(\pi)}$ constructed in Lemma \ref{lem:construct1}. 
 %is a $C_{d}$-extension of ramification index $e$. 
%Now replace $E/K(t)$ by an extension ${E}/K(t)$ having $C_{d}$-extensions of ramification index $e$ as completions at suitable places as above, for {\it all} $d\divides P$ and $e \divides d$ with $e>1$ (this is still possible due Saltman's theorem).

%We now claim (*) holds for $E/K(t)$. 
\DN{Recall that every prime $\fp$ of $K$ for which $K_\fp$ has a tame $C_{d}$-extension of ramification index $e$ is such that $\zeta_{e}\subseteq K_\fp$, for each $e> 1$. 
%In particular, denote by $\mathcal{S}_e)$ the set of primes of $K$ for which $\mu_{P^\infty}\cap K_\fp = \mu_{e}$, for each $1\le e\le f$.
%For any fixed $\pi=(e',d)$ as in (*), the union of the sets  $\mathcal{S}_{e}^{(d)}$, over all $e$ which divide $d$ and are divisible by $e'$,  contains all primes $\fp$ for which $K_\fp$ possesses a tame $C_{d}$-extension of ramification index $e'$.
Write $\mu_d\cap K_\fp=\mu_{e'}$ so that $e\divides e'$ and $\fp\in \mathcal S_{e'}^{(d)}$, and set $\pi=(e',d)$. 
 Since $E^{(\pi)}/F^{(\pi)}$ is a completion of $E/K(t)$, Lemma \ref{lem:spec-odd} then implies that every $C_d$-extension of $K_\fp$ with ramification index $e$ (dividing $e'$) is a specialization of $EK_\fp/K_\fp(t)$, so that (*) holds. }  % of $E^{(\pi)}/F^{(\pi)}$ for . %since for every pair $\pi=(e',d)$ as in (*) that satisfies $\mu_d\cap K(\mu_{e')}=\mu_e$, and all $\fp\in \mathcal{S}_{e}^{(d)}$, the extension ${E}/K(t)$, which has $E^{(\pi)}/F^{(\pi)}$ as a completion, achieves condition (*). 
%for all $(e,d)$ with $e\divides e'$. 
%To cover (*) fully, for a pair $(e',d)$ write $\mu_d\cap K(\mu_{e'})$  as in (*), let $\mu_e$ By definition of the sets $\mathcal{S}_e^{(d)}$, it now follows readily that ${E}/K(t)$ yields condition (*) in full.

By adding finitely many more analogous completions with $P$ replaced by $Q$, we may assume that it also achieves (**).

Finally, to fulfill also condition (***), we  add for each $f,e$ as in (***), the $C_e\rtimes C_f$-extension constructed in \cite[Corollary 3.6]{KN} to the completions of $E/K(t)$. \cite[Corollary 3.11]{KN} then implies that every $(C_e\rtimes C_f)$-extension of $K_\fp$ with ramification index $e$ is a specialization of \JK{$EK_\fp/K_\fp(t)$}, for all but finitely many primes $\fp$ of $K$. 
\end{proof}

\iffalse
\begin{remark}
\marginpar{\JK{In favor of dropping this remark (not fully clear, and paper's long enough).}}
The proof of Theorem \ref{thm:cyclic_odd} actually shows that, for all the groups $G$ under consideration, one may choose the extension ${E}/K(t)$ as a compositum ${E}'\cdot K/K(t)$, where ${E}'/\mathbb{Q}(t)$ is a $G$-extension over $\mathbb{Q}$ with the Hilbert--Grunwald property, and depends only on $G$, not on $K$ (indeed, if anything, some of the local extensions used in the construction over $\mathbb{Q}(t)$ may become superfluous over $K(t)$)! In the language of covers, there exists a Galois cover $f:X\to \mathbb{P}^1$ with group $G$ over $\mathbb{Q}$ whose specializations provide solutions to all Grunwald problems outside a finite set of primes, not only for $\mathbb{Q}$, but over any number field, thus emphasizing the special role of $K=\mathbb{Q}$ in Question \ref{ques:hg2}.
\end{remark}
\fi
%\begin{remark}\label{rem:proof-use}
%The proof furthermore constructs a $C_{p^f}$-extension of $K(\zeta_{p^e})((y))$ with ramification index $p^e$ such that every extension $E/K(t)$ that admits this extension as a completion specializes to every $C_{p^f}$-extension of $K_v$ with ramification index $p^e$ for all but finitely many places $v$ of $K$ if $p$ is odd, and for all but finitely many primes $v$ of norm $N(v)\equiv 1$ (mod $4$) if $p=2$. 
%\end{remark} 
We can also conclude a proof of Theorem \ref{thm:qi}. As the forward direction was shown earlier in the section, it suffices to show:
\begin{proof}[Proof of Theorem \ref{thm:qi}, converse assertion]
Let $\fp$ be an odd prime. We first claim that $K_\fp$ has no extensions whose Galois group a generalized quaternion group. 
Since $\sqrt{-1}\in K\subseteq K_\fp$, 
and the action of Frobenius on the inertia group is equivalent to its action on roots of unity, cf.\ Section \ref{sec:local}, 
%one has $q=N(\fp)\equiv 1$ mod $4$, and hence 
the quaternion group does not appear as a Galois group over $K_\fp$ or any finite extension of it, and hence the generalized  quaternion groups are not Galois groups over $K_\fp$.
%quotients of the Galois group $\langle \sigma,\tau\,|\,\sigma^{-1}\tau\sigma=\tau^q\rangle$ of the maximal tamely ramified extension of $K_\fp$. 
%In particular, $K_\fp$ has no tamely ramified extension whose Galois group is a generalized quaternion group.

Thus, whether  $G$ is a generalized quaternion group or  $G=C_P\rtimes C_Q$ for coprime prime powers $P$ and $Q$, as in the proof of Theorem \ref{thm:hgp}, it suffices to construct an extension $E/K(t)$ fulfulling: 
\begin{enumerate}
    \item[(*)] For each $(e,d)$ such that $e>1$, $e\divides d$, and $d\divides P$, and for all but finitely many primes $\fp$ of $K$ such that $K_\fp$ possesses a $C_d$-extension of ramification index $e$, all $C_d$-extensions of $K_\fp$ with ramification index $e$ are in $Sp({E}\cdot K_\fp/K_\fp(t))$. 
    \item[(**)] The same with $P$ replaced by $Q$. 
    \item[(***)] For every pair $(e,f)$ such that $\JK{e,f}>1$, $e\divides P$, and $f\divides Q$, and for all but finitely many primes $\fp$ of $K$ such that $K_\fp$ posses a $(C_e\rtimes C_f)$-extension of ramification index $e$, each such extension is in $Sp(E\cdot K_\fp/K_\fp(t))$. 
\end{enumerate}
Since $G$ has a generic polynomial over $K$, once again we construct $E/K(t)$ using \cite[Theorem 5.9]{S82} so that it has the following completions at distinct places. For each $(e,d)$ as in (*) and (**) satisfying $\mu_d\cap K(\zeta_e)=\mu_e$, $E/K(t)$ has a completion isomorphic to the extension $E^{(e,d)}/F^{(e,d)}$ from Lemma \ref{lem:construct1}. (Note that it is possible to apply Lemma \ref{lem:construct1} since $\sqrt{-1}\in K$). For each $(e,d)$ as in (***), $E/K(t)$ should admit a completion isomorphic to the $(C_e\rtimes C_f)$-extension with ramification $e$ constructed in \cite[Corollary 3.6]{KN}.  %isomorphic to $K(\zeta_e)((t^{1/e}))/F((t))$, where $K\subseteq F\subseteq K(\zeta_e)$ is unique index $f$ intermediate field. 
As in the proof of Theorem \ref{thm:hgp},  Lemma \ref{lem:spec-odd} (resp.\ \cite[Corollary 3.11]{KN})  implies that $E/K(t)$ has  (*), (**) (resp.\ (***)), as desired.  
\end{proof}
Finally, we show that there are no further  obstructions to $\ld_K(G)=1$ for cyclic groups $G$ when $K$ contains the $p$-th roots of unity for every prime $p\divides |G|$, under the further assumption that $K(\mu_{2^\infty})/K$ is  cyclic if $|G|$ is even. 
%Thus, 
%Under the above assumptions, the only remaining obstruction
%in order for $G$ to have the Hilbert--Grunwald property over $K$ it is necessary that  $\zeta_p\in K$ for all $p\divides |G|$, by  Lemma \ref{cyclic_dec_group2}. In the above setting, we show that this is the only obstruction:
%In order to have the Hilbert--Grunwald property in the former case (resp.\ latter case), Lemma \ref{cyclic_dec_group2} (resp.\ Lemma \ref{cyclic_dec_group3}) implies that $\zeta_p\in K$ (resp.\ $\sqrt{-1}\in K$) for every $p\divides |G|$. 
\begin{theorem}\label{thm:main-cyclic}
%Or Proposition? 
Let $K$ be a number field, 
and $G$ a cyclic group 
%with a generic extension over $K$ 
such that $\zeta_{p}\in K$ for every prime $p\divides |G|$.
If $G$ is of even order, further assume that $K(\mu_{2^\infty})/K$ is cyclic. 
%, and moreover $\sqrt{-1}\in K$ if $4\divides |G|$.  %with $C_{pq}\leq G$. 
Then there exists a $K$-regular $G$-extension $E/K(t)$ with the Hilbert--Grunwald property. In particular,  $\ld_K(G)=1$. 
\end{theorem}
We note that the proof uses the cyclicity of  $K(\mu_{2^\infty})/K$ only in order to ensure that $G$ has a generic extension. This in turn is used only to assert the existence of an extension $E/K(t)$ with desired local completions. It is possible that such an extension exists even when $G$ has no generic extension. 

%If it is  It in fact suffices to assume the existence of an extension $E/K(t)$ having the above extensions $E_{e,d}/F_{e,d}$ and $E_j^{(e,d)}/F_j^{(e,d)}$, $j=1,2,3$ as completions.

%We note that the cyclic group $C_{2^s}$ has a generic extension  when $K(\zeta_{2^s})/K$ is cyclic, so that every cyclic group whose $2$-Sylow subgroup is of exponent $\leq 2^s$ has a generic extension over such field $K$.

Since the case of cyclic groups of odd prime power order (or of order $2$) is treated in Theorem \ref{thm:cyclic_odd}, in the following we assume either that $|G|$ is not a prime power or that $|G|$ is a power of $2$.

We start the proof of Theorem \ref{thm:main-cyclic} by constructing the completions of $E/K(t)$:
%the parametrizing  extension:

\begin{prop}\label{prop:cyclic}
Let $e,d$ be positive integers with the same prime divisors, such that $v_2(d)\geq 2$ and $v_2(e)=1$.
Let $K$ be a totally complex number field containing $\zeta_p$ for all primes $p\divides e$. %\geq 2$ and $e,d$ odd. 
Then there exist complete fields $F_j=F_{j}^{(e,d)}$ and $C_{d}$-extensions $E_j^{(e,d)}/F_j$, $j=1,2,3$ with ramification index $e$ and residue extensions $N_j/M_j$, $j=1,2,3$ satisfying: \\ (i) $M_j$, $j=1,2,3$ are contained in a common biquadratic extension of $K(\zeta_e)$, and \\
(ii) $N_j$ contains the quadratic extension $M_j(\sqrt{-1})/M_j$ and the extension  $M_j(\zeta_{d'})/M_j$ where $d'=d/2^{v_2(d)}$, for $j=1,2,3$. 
\end{prop}
We first construct the residue fields of the above completions: %extensions:
\begin{lemma}\label{lem:2-constr}
Let $K$ be a number field such that $\sqrt{-1}\notin K$, and let $r\geq 2$ be an integer. 
%and $L/K$ a finite extension. % linearly disjoint from $K(i)/K$. 
Then there exists a biquadratic extension $M/K$ such that for each of its quadratic subextensions $M_j/K$, $j=1,2,3$, the extension $M_j(\sqrt{-1})/M_j$ is a quadratic extension which embeds into a $C_{2^r}$-extension. % which is linearly disjoint from $LM_j(i)$ over $M_j(i)$. 
\end{lemma}
\begin{proof}
%Let $\zeta_s$ be a primitive $2^s$-th root of unity, 
%$\mu_\infty=\bigcup_{s=1}^\infty \langle\zeta_s\rangle$.
Let $s=s(K)$ be the largest integer for which $\eta_s:=\zeta_{2^s}+\zeta_{2^s}^{-1}$ is in $K$. Let $S_0(K)$ be the set of even primes $\fp$ of $K$ for which $K_\fp(\zeta_{2^{s+1}})/K_\fp$ is non-cyclic.

By the Grunwald--Wang theorem\footnote{Note that the special case of Grunwald--Wang does not apply to quadratic extensions, and those are the only ones considered here.} \cite[Chp.\ X, Thm.\ 5]{AT}, there exist two disjoint quadratic extensions $M_1,M_2/K$  whose completions satisfy:
\begin{enumerate}
    \item[(i)] $(M_1)_\fp$ and $(M_2)_\fp$  are disjoint quadratic extensions with compositum $K_\fp(\zeta_{2^{s+1}})$, at every $\fp\in S_0(K)$. 
    \item[(ii)]  there exist odd  primes $\fp_1$ and $\fp_2$ which split completely in $K(\zeta_{2^{s+1}})$ and are inert in $M_1$ and $M_2$, respectively
\end{enumerate}
Let $M=M_1\cdot M_2$, and let $M_3$ be the third quadratic subextension of $M/K$. The disjointness of $M/K$ and $K(\zeta_{2^{s+1}})/K$  follows from (ii) since in each of the extensions $M_j/K$, $j=1,2,3$, one of the primes $\fp_1,\fp_2$ is inert while being split completely in $K(\zeta_{2^{s+1}})/K$.

We claim that for each of the quadratic fields $M'=M_j$, $j=1,2,3$ the extension $M'(\sqrt{-1})/M'$ embeds into a $C_{2^r}$-extension $N'/M'$. Setting $\Gamma_\nu:=\Gal(M_\nu'(\sqrt{-1})/M_\nu')$, we first check that the corresponding local embedding problem $\pi_\nu:\pi^{-1}(\Gamma_\nu)\ra \Gamma_\nu$  is solvable at all places $\nu$ of $M'$. 
Since $M'$ is totally complex, $M'(\sqrt{-1})/M'$ is unramified away from even primes, and hence $\pi_\nu$ is solvable for every place $\nu$ which does not lie over $2$.
Next note that since $M/K$ is disjoint from $K(\zeta_{2^{s+1}})/K$ while $M_\fp/K_\fp$ and $K_\fp(\zeta_{2^{s+1}})/K_\fp$ are not disjoint for  $\fp\in S_0(K)$ by (i), we have $s(M')=s(K)$ and $S_0(M')=\emptyset$. 
Since $S_0(M')=\emptyset$, for every even prime $\fp$ of $M'$, either $M'_\fp(\mu_{2^\infty})/M_\fp'$ is cyclic or $\eta_{s+1}\in M'_\fp$, where $s=s(M')$. In the former case the embedding problem $\pi_\fp$ is clearly solvable. In the latter case $M'_\fp\supseteq \mathbb Q_2(\sqrt{2})$. Since the quaternion algebra $(-1,-1)_{\qq_2}$ is split by the extension $\qq_2(\sqrt{-1})/\qq_2$, one has $-1=a^2+b^2$ for some\footnote{Alternatively, this follows from the congruence $-1\equiv 1^2+(2+\sqrt{2})^2$ mod $4\sqrt{2}$ in $\mathbb Z_2[\sqrt{2}]$ and Hensel's lemma.}  $a,b\in M'_\nu$. 
Since $-1$ is a sum of two squares,  $M'_\fp(\sqrt{-1})/M'_\fp$ embeds into a $C_4$-extension \cite[Proposition 16.5.1]{FJ}, which further implies that $M_\fp'(\sqrt{-1})/M'_\fp$ embeds into a $C_{2^r}$-extension  \cite[Thm.\ 3]{GJ} or \cite[Proof of Thm.\ 5]{AFSS}. In summary, $\pi_\nu$ is solvable for all places $\nu$ of $M'$. Finally, as in \cite[$17^\circ$.(i)]{GJ}, since $M'/K$ is linearly disjoint from $K(\zeta_{2^{s+1}})$ and since $\pi:C_{2^r}\ra \Gal(M'(\sqrt{-1})/M')$ is locally solvable at all places, $\pi$ is  solvable by  \cite[Theorem 2]{GJ}. Since $\pi$ is Frattini, every  solution is proper, as desired. 
%Finally we claim that the solution field $N'/M'$ and the compositum $LM'$ are linearly disjoint over $M'$. Indeed, since $M'$ and $L(i)$ are linearly disjoint over $K$, and so are $K(i)$ and $L$,  the compositums $M'(i)$ and $LM'$ are linearly disjoint over $K$, and hence so are $N'$ and $LM'$,   proving the claim. 
\end{proof}
%{\bf Joachim: Does case (*) of this construction generalize to arbitrary (say, odd order) cyclic group $C_N$, under the assumption that $p$-th roots of unity are contained for all prime divisors $p$ of $N$? This would be necessary to get a complete classification for abelian groups over totally complex fields.}

\begin{proof}[Proof of Proposition \ref{prop:cyclic}]
Note that $r:=v_2(d)>1$, and set $d':=d/2^r$ and $e':=e/2$. 
By Lemma \ref{lem:2-constr}, there exist quadratic extensions $M_j/K(\zeta_e)$, $j=1,2,3$ contained in a common biquadratic extension of $K(\zeta_e)$, and $C_{2^r}$-extensions $N_j'/M_j$ containing the quadratic extension $M_j(\sqrt{-1})/M_j$. 
Note that since $d'$ and $e'$ have the same prime divisors,   $K(\zeta_{d'})/K(\zeta_e)$ and $M_j/K(\zeta_e)$ are linearly disjoint, and the condition $\mu_{d'}\cap M_j=\mu_{e'}$ holds. Thus we may apply Lemma \ref{lem:construct1} with the base field $M_j$ (instead of $K$) to obtain a $C_{d'}$-extension $E_j^{(e',d')}/F_j$, with ramification index $e'$ and residue extension $M_j(\zeta_{d'})/M_j$, of a complete field $F_j=M_j((y_j))$.

%Since $\zeta_p\in K$ for every odd prime $p\divides e$, the extension $K(\zeta_{d'})/K(\zeta_e)$ is Galois of odd degree, and hence  linearly disjoint from $N_j'/K(\zeta_e)$. 
%which are linearly disjoint from $M_j(i,\zeta_{e})$ over $M_j(i)$. 
%Thus $N_j'/M_j$ is linearly disjoint from 
%$M_j(\zeta_e)/M_j$, and  $E_j':=N_j'E_j^{(e',d')}$ is a cyclic extension of  of degree $d=2^rd'$ and ramification index $e'$.  
%We construct $M/K$ so that
%\begin{itemize}
    %\item 
%    which is disjoint from $K(\zeta_{s+1})$, and such that $M_\fp/ K_\fp$ is biquadratic for all $\fp\in S_0$.
    
%Let $N'_j/M'_j$ be a $C_{2^r}$-extension containing $M_j(i)/M_j$, $j=1,2,3$. To construct $E_j/F_j$, choose a complete field $F_j$ with residue field $M_j$ and uniformizer $m_j\in F_j$. 
Write $N_j'=N''_j(\sqrt{a_j})$ for the $C_{2^{r-1}}$-subextension $N''_j/M_j$ of $N_j'/M_j$, and $a_j\in N_j''$. 
Then $N_j''F_j(\sqrt{a_jy_j})$ is a $C_{2^r}$-extension of $F_j$ with ramification index $2$ which is linearly disjoint from  $E_j^{(e',d')}/F_j$, as the latter is Galois of odd degree.  Then $E_j:= E_j^{(e',d')}N_j''(\sqrt{a_jy_j})$ is a $C_{d}$-extension with ramification index $e=2e'$, containing $N_j''
\supseteq M_j(\sqrt{-1})$ and $M_j(\zeta_{d'})$, as desired.  
\end{proof}

\begin{lemma}\label{lem:2-spec}
Let $e\divides d$ be positive integers with the same prime divisors such that $\mu_d\cap K(\zeta_e)=\mu_e$ and $v_2(d)>v_2(e)=1$. 
Let $E/K(t)$ be an extension with the above completions %$t\mapsto t_j$ is the above constructed extension 
$E_j^{(e,d)}/F_j^{(e,d)}$, $j=1,2,3$. Then every  $C_d$-extension of $K_\fp$ with ramification index dividing $e$ is a specialization of $EK_\fp/K_\fp(t)$ for all but finitely many primes $\fp$ satisfying $\mu_d\cap K_\fp = \mu_e$.  %$e=\gcd(N(\fp)-1,d)$. 
%/K(t)$ for some $j$,  for all but finitely many primes $\fp$ of $K$ with $N(\fp)\equiv 3$ (mod $4$). 
\end{lemma}
\begin{proof}
%We next show that theses extensions $E_j/F_j$, $j\in I$ have the required specialization property: 
%Note that since $4\divides d$ and $2$ is the largest $2$-power dividing $\gcd(N(\fp)-1,d)$,
%we in fact have $N(\fp)\equiv 3$ mod $4$. 
As always, avoiding finitely many primes, we may assume $\fp$ is coprime to $e$. Since $4\divides d$, $v_2(e)=1$ and $\mu_d\cap K_\fp=\mu_e$, we deduce that $K_\fp$ does not contain $\sqrt{-1}$. 
It follows that there is no  $C_{d}$-extension of $K_\fp$ with  ramification index divisible by $4$. 

%By D\`ebes--Ghazi \cite{DG}, every unramified $C_{2^s}$-extension of $K_\fp$ is a specialization of $E/K(t)$. It remains to show that, for all but finitely many such primes $\fp$, every extension $T_\fp/K_\fp$ with ramification index $2$ is a specialization of $E/K(t)$. Note that there is a unique $C_{2^s}$-extension $T_\fp/K_\fp$ with ramification index $2$. 

Letting $M_j$  denote the residue field of $F_j=F_j^{(e,d)}$, we recall that  $M_j/K(\zeta_e)$, $j=1,2,3$ are quadratic extensions contained in a common biquadratic extension  $M/K(\zeta_e)$. Avoiding the finitely many primes ramified in $M$ and recalling the $\fp$ splits completely in $K(\zeta_e)$, a prime $\fp'$ of $K(\zeta_e)$ lying over $\fp$ has to split \footnote{\DN{This idea is inspired by a construction of polynomials with roots modulo $\fp$ for all $\fp$, cf.\ \cite{Sonn}.}} in one of the quadratic fields $M_j$, $j=1,2,3$,  say in $M_1$. Let $\fP$ be a prime of $M_1$ lying over $\fp'$. Moreover, since the residue extension $N_1/M_1$ of $E_1/F_1$ contains $M_1(\zeta_{d'})$, where  $d'=d/2^{r}$ and $r=v_2(d)$, its residue degree at $\fP$  is divisible by $d'/(e/2)$. Moreover,   $N_1/M_1$ contains a $C_{2^{r-1}}$-extension containing   the quadratic subextension $M_1(\sqrt{-1})/M_1$. Since $\fP$ is inert in $M_1(\sqrt{-1})/M_1$, it is inert in the above $C_{2^{r-1}}$-extension, and hence the residue degree of $N_1/M_1$ at $
\fP$ is at least $2^{r-1}\cdot d'/(e/2) = d/e$. Since $E_1/F_1$ is a $C_d$-extension with ramification index $e$, this forces the residue degree to equal $d/e$.

Since $\fp$ has a degree $1$ prime over it in $M_1$, as in the proof of Lemma \ref{lem:spec-odd},  Proposition \ref{mainlemma_kln}iii)   gives (infinitely many) $t_0\in K$ such that the completion $E_{t_0}/K$ at $\fp$ is ramified of index $e$ and its residue degree is the same as that of $N_1/M_1$ at $\fP$, which by the above is  $d/e$. 
Proposition \ref{thm:KLN2} then shows that furthermore  {\it all}   $C_d$-extensions with ramification index $e'$ dividing $e$ are specializations of $E\cdot K_\fp/K_\fp(t)$. 
%This holds for all primes $\fp$ for which $\mu_d\cap K_\fp = \mu_e$ since such primes $\fp$ split completely in $K(\zeta_{e})$ and the prime of $K(\zeta_e)$ lying above $\fp$ remain inert in $K(\zeta_{d})/K(\zeta_{e})$.
%
% More basic proof in the sense that it does not use Theorem 4.4 of KLN
%Since ${(M_1)}_\fp\cong K_\fp$, we may pick $t_0\in K_\fp\setminus \{t_j\}$ which is $\fp$-adically close to $t_1$, so that for all but finitely many primes $\fp$ the specialization $E_{t_0}/K_\fp$ has ramification index $2$ by Proposition \ref{beckmann}, and its residue at $\fp$ is the same as that of $N_1/M_1$ (the residue at $t_1$) by Proposition \ref{mainlemma_kln}. 
%However, since by our choice of $\fp$,  $K_\fp(\sqrt{-1})/K_\fp$ is a quadratic extension, and since $N_1\supseteq M_1(\sqrt{-1})\supseteq M_1$, it follows that the residue of $N_1/M_1$ at $\fp$, and hence that of $E_{t_0}/K_\fp$ is a $C_{2^{s-1}}$-extension. Thus, $E_{t_0}/K_\fp$ coincides with the unique $C_{2^s}$-extension $T_\fp/K_\fp$ with ramification index $2$ as desired. 
%the specialization at $E_{t_0}/K_\fp$ has  
\end{proof}
%The following corollary shows that for abelian groups with a generic extension the  obstructions given in Section \ref{sec:hgp} are the only ones. 

%\begin{kor}
%Let $G$ be a nontrivial cyclic group and $K$ a totally complex number field. Then $\ld_K(G)=1$ if and only if $\zeta_{pq}\in K$ for all pairs of distinct primes $p,q$ dividing $|G|$.
%\end{kor}
\begin{proof}[Proof of Theorem \ref{thm:main-cyclic}]
Consider the pairs $(e,d)$ of positive integers having the same prime divisors, such that $e\divides d$,  and $d\divides |G|$, and $\mu_d\cap K(\zeta_e)=\mu_e$. Since $G$ has a generic polynomial over $K$, Saltman's theorem \cite[Theorem 5.9]{S82} implies that there exists a $G$-extension $E/K(t)$ having completions isomorphic to the extensions $E^{(e,d)}/F^{(e,d)}$ of Lemma \ref{lem:construct1} for each of the above  $(e,d)$ satisfying $v_2(e)=v_2(d)\leq 1$ or $v_2(e)>1$, and to the extensions $E_j^{(e,d)}/F_j^{(e,d)}$, $j=1,2,3$ of Lemma \ref{lem:2-constr} for each of the above $(e,d)$  satisfying  $v_2(e)= 1<v_2(d)$. By Lemmas \ref{lem:spec-odd} and   \ref{lem:2-spec}, every cyclic extension of $K_\fp$ of degree $d$ and ramification index dividing $e$ \JK{is a specialization of $E\cdot K_\fp/K_\fp(t)$} for all but finitely many primes $\fp$ with $\mu_d\cap K_\fp=\mu_e$. %$\gcd(N(\fp)-1,d)=e$. 

It remains to recall that the primes $\fp$ for which $K_\fp$ admits a tamely ramified $C_d$-extension with ramification index $e'$ are those for which $\mu_d \cap K_\fp = \mu_e$ with $e'\divides e$. Thus, the above implies that $E/K(t)$ specializes to all  $C_d$-extensions of $K_\fp$ for all but finitely many primes $\fp$. 
%For each pair $(e,d)$
%Note that every cyclic degree $d$ extension with ramification $e'$ of $K_\fp$, with $d\divides |G|$, one has $N(\fp)\equiv 1$ mod $e'$, and hence $e'\divides e=\gcd(N(\fp)-1,d)$. 
%Since $\mu_p\subseteq K$, we in fact have $N(\fp)\equiv 1$ mod $p$ for every prime $p\divides |G|$, and hence $e$ and $d$ have the same prime divisors. It hence follows from the above paragraph that every such extension is a specialization of $E/K(t)$. 
%with ramification index $e'$
%the largest $p$-power $p^{m_p}$ dividing
%Apply here Remark \ref{rem:proof-use}
\end{proof}

\appendix
\section{An application: Groups of parametric dimension $1$}
Our results also allow a near-full classification of groups of {\it parametric dimension} $1$ over a number field, a problem that has been investigated in a series of previous papers. (e.g., \cite{KL18}, \cite{KLN}). Here, the parametric dimension $\textrm{pd}_K(G)$ of a group $G$ over any field $K$ is defined as the minimal %transcendence degree of a $G$-extension of function fields over $K$ specializing to every Galois extension of $K$ with Galois group embedding into $G$.\footnote{In some of the precursor papers, the analogous notion with Galois extensions of $K$ of group {\it exactly} $G$ was considered.
integer $d$ for which there exist finitely many $G$-extensions (of \'etale algebras) $E_i/F_i$, $i=1,\dots, r$, of transcendence degree $\le d$ over $K$ such that every (\'etale algebra ) $G$-extension of $K$ (equivalently, every Galois extension of $K$ with Galois group embedding into $G$) occurs as a specialization of some $E_i/F_i$.
%Indeed, the extra requirement to reach also all sub-$G$-extensions is very restrictive, and in many cases makes a proof of the assertion $\textrm{pd}_K(G)>1$ much easier. Still, we are not aware of a direct argument of the generality reached in this section.}

The following connects parametric and local dimension:
\begin{prop}
For all groups $G$ and number fields $K$, one has $\textrm{pd}_K(G) \ge \textrm{ld}_K(G)$.
\end{prop}
%\marginpar{Insert proof. Uses solvability of Grunwald problems for metacyclic groups (Harpaz-W.)}
\begin{proof}
Let $E_i/F_i$, $i=1,\dots, r$, be finitely many $G$-extensions of transcendence degree $d\le \textrm{pd}_K(G)$, parameterizing all $G$-extensions of $K$ as in the definition above. In particular, for every metacyclic subgroup $U\le G$, every Galois extension $L/K$ with Galois group $U$ occurs as a $K$-specialization of some $E_i/F_i$. A fortiori, $L\cdot K_\fp/K_\fp$ occurs as a $K_\fp$-specialization of $E_i/F_i$, for all primes $\fp$ of $K$. Note that metacyclic groups are necessarily supersolvable, whence by a famous result of Harpaz and Wittenberg (\cite{HW}), for all but finitely many primes $\fp$ of $K$, all $U$-extensions of $K_\fp$ do in fact occur as completions of some $U$-extension of $K$. On the other hand, all tame Galois extensions of $K_\fp$ are in fact metacyclic. Therefore, up to possibly excluding finitely many more primes $\fp$, every Galois extension of $K_\fp$ of Galois group embedding into $G$ is a $K_\fp$-specialization of some $E_i/F_i$, for all but finitely many $\fp$. Thus, $\textrm{pd}_K(G) \ge \textrm{ld}_K(G)$.
\end{proof}

In particular, groups of parametric dimension $1$ over e.g. $K=\mathbb{Q}$ are necessarily of the form $C_P\rtimes C_Q$ as in Theorem \ref{thm:hgmain}. We may narrow the list down even further, by combining this result with the methods of \cite{KL18}. In particular, \cite[Theorem 5.2]{KL18} excludes all cyclic groups except the ones of prime order. Note furthermore that, for $U\le G$, one has $\textrm{pd}_K(U) \le \textrm{pd}_K(G)$ by simply applying the Galois correspondence. We are thus even reduced to groups of the form $C_p\rtimes C_q$ with both $p$ and $q$ prime (or $=1$), and $q|p-1$. Furthermore, \cite[Theorem 5.1]{KL18} excludes all groups of order coprime to $6$. 

\JK{\begin{remark} Invoking \cite{KL18} as above requires a certain amount of care; indeed, the cited theorems are a priori stated only for $\mathbb{Q}$-regular extensions $E_i/\mathbb{Q}(t)$ with purely transcendental base field $\mathbb{Q}(t)$, whereas we here also need to allow arbitrary extensions of transcendence degree $1$ function fields. This is, however, possible via slight adjustments of the proofs in \cite[Section 6]{KL18}. Indeed, the methods apply in complete analogy for $\mathbb{Q}$-regular extensions $E_i/F_i$ with arbitrary base field $F_i$. To drop also the regularity assumption, start by picking, for each non-regular extension $E_i/F_i$, a prime which is non-split in the constant extension, and restrict attention to only those $G$-extensions of $\mathbb{Q}$ in which all those primes split. Obviously, no such extension can arise as a specialization of some non-regular $E_i/F_i$, so from here on, we may again assume that all $E_i/F_i$ are $\mathbb{Q}$-regular. Now the core argument of the cited results of \cite{KL18} is the fact that (for the groups $G$ in question) there is some non-trivial normal subgroup $N\triangleleft G$ such that every properly solvable embedding problem induced by $G\to G/N$ has infinitely many proper solutions, whereas the given set of regular extensions can only specialize finitely many of those. It then suffices to construct (infinitely many) solutions in which the finitely many primes chosen above are all split; since all groups $G$ in question are solvable, this can be obtained, e.g.\ from Shafarevich's method, see, e.g., \cite{FL}.
\end{remark}}

We have thus obtained:
\begin{theorem}
\label{thm:param}
Let $G$ be a finite group with $\textrm{pd}_\mathbb{Q}(G)=1$. Then $G$ is isomorphic to $C_p$ (for some prime $p$), $D_p$ (for some prime $p\ge 3$) or $C_p\rtimes C_3$ (for some prime $p\equiv 1$ mod $3$).
\end{theorem}

\JK{Note also that comparison of the lists of groups in Theorems \ref{thm:hgmain_strong} and \ref{thm:param} gives examples of groups for which $\textrm{pd}_\mathbb{Q}(G) > \textrm{ld}_\mathbb{Q}(G)$, something that was anticipated, but not proven in \cite[Appendix A]{KN}.}

In fact, we conjecture that the only groups $G$ with $\textrm{pd}_\mathbb{Q}(G)=1$ are $G=C_2, C_3$ and $S_3$. Combining Theorem \ref{thm:param} with the methods of \cite{KL19} would get us much closer to this goal, at least conditionally on the abc-conjecture, although going into full detail would lead us too far away.

\end{document}